\documentclass{amsart}
\usepackage{soul}

\usepackage{amsmath,amstext,amssymb,mathrsfs,amscd,amsthm}
\usepackage{amsfonts}
\usepackage[all]{xy}
\usepackage{booktabs}
\usepackage{verbatim}
\usepackage{enumitem}
\usepackage{mathtools}
\usepackage{color}

\usepackage{graphicx}                   
\usepackage{subfig}
\usepackage{xspace}

\newtheorem{lemma}{Lemma}[section]
\newtheorem{proposition}[lemma]{Proposition}
\newtheorem*{theorem}{Theorem}

\newtheorem{corollary}[lemma]{Corollary}

\newtheorem{predf}[lemma]{Definition} 
\newenvironment{df}{\begin{predf}\rm}{\end{predf}}
\newtheorem{preremark}[lemma]{Remark}  
\newenvironment{remark}{\begin{preremark}\rm}{\end{preremark}}
\newtheorem{preremark0}[lemma]{Remark}  

\newtheorem{preexample}[lemma]{Example}
\newenvironment{example}{\begin{preexample}\rm}{\end{preexample}}
\newtheorem*{prenotation}{Notation}

\newtheorem*{preproblem}{Problem}

\newtheorem*{preconjecture}{Conjecture}

\newtheorem{atheorem}{Theorem}

\def\bichord{bichord\xspace}

\def\jalmax{j_{\mathrm{almax}}}
\renewcommand\flat[1]{|\pi_0 G(#1)|}

\numberwithin{equation}{section}

\def\s{\mathfrak{s}}

\def\btoab{\cA}

\newcommand\presh[2]{#2^{#1}}
\newcommand\Mod[1]{\text{$#1$-Mod}}

\def\Ch{\mathrm{Ch}}
\def\Sp{\mathbf{Sp}}

\def\u{\mathfrak{u}}

\def\la{\leftarrow}

\def\Deltainj{\Delta_{\mathrm{inj}}}
\def\Deltainjaug{\Delta_{\mathrm{inj*}}}

\def\op{\mathrm{op}}
\def\source{s}
\def\target{t}

\def\Tot{\operatorname{Tot}}

\def\Setp{\mathbf{Set}_\bullet}
\def\Set{\mathbf{Set}}

\def\Topp{\mathbf{Top}_\bullet}
\def\Top{\mathbf{Top}}
\def\boldast{\mathbf{\ast}}
\def\Kh{Kh}

\def\jalmax{{j_{\mathrm{almax}}}}

\def\G{\mathcal{G}}
\makeatletter
\renewcommand{\boxed}[1]{\text{\fboxsep=.2em\fbox{\m@th$\displaystyle#1$}}}
\makeatother
\newcommand{\cube}[1]{\mathbf{2^{#1}}}

\def\cero{\vec{0}}
\def\uno{\vec{1}}

\newcommand{\bN}{\mathbb{N}}

\newcommand{\bR}{\mathbb{R}}

\newcommand{\bZ}{\mathbb{Z}}

\newcommand\lra{\longrightarrow}
\newcommand\lla{\longleftarrow}

\newcommand\hocolim{\operatorname*{hocolim}}

\def\Id{\mathrm{Id}}

\newcommand{\hocofib}{\ensuremath{\mathrm{hocofib}}}

\renewcommand{\geq}{\geqslant}
\renewcommand{\leq}{\leqslant}
\newcommand{\cA}{\mathcal{A}}
\newcommand{\cB}{\mathcal{B}}
\newcommand{\cC}{\mathcal{C}}
\newcommand{\cD}{\mathcal{D}}

\newcommand{\cL}{\mathcal{L}}

\newcommand{\cO}{\mathcal{O}}

\newcommand{\cX}{\mathcal{X}}
\newcommand{\cY}{\mathcal{Y}}

\def\sd\operatorname{\mathrm{Sd}}

\title{Almost-extreme Khovanov spectra}
\author{Federico Cantero Mor\'an and Marithania Silvero}
\thanks{Both authors were supported by the Spanish projects MTM2016-76453-C2, MDM-2014-0445 and FEDER. The first author was partially supported by project PID2019-108936GB-C21. The second author was partially supported by project US-1263032.}
\email{federico.j.cantero@gmail.com, marithania@us.es}
\address{Departamento de Matem\'aticas, Universidad Aut\'onoma de Madrid and ICMAT, Madrid, Spain}
\address{Departamento de Matem\'atica Aplicada, Universidad de Sevilla and IMUS, Sevilla, Spain}

\begin{document}
\begin{abstract}
We introduce a functor from the cube to the Burnside $2$-category and prove that it is equivalent to the Khovanov spectrum given by Lipshitz and Sarkar in the almost-extreme quantum grading. We provide a decomposition of this functor into simplicial complexes. This decomposition allows us to compute the homotopy type of the almost-extreme Khovanov spectra of diagrams without alternating pairs.
\end{abstract}
\maketitle

\section{Introduction}
Khovanov homology is a powerful link invariant introduced by Mikhail Khovanov in \cite{Khovanov00} as a categorification of the Jones polynomial. More precisely, given an oriented diagram $D$ representing a link $L$, he constructed a finite $\bZ$-graded family of chain complexes
\[\xymatrix{\ldots\ar[r]& C^{i,j}(D)\ar[r]^-{d_i} & C^{i+1,j}(D)\ar[r]^-{d_{i+1}} & C^{i+2,j}(D)\ar[r]&\ldots}\]
whose bigraded homology groups, $\Kh^{i,j}(D)$, are link invariants. The groups $\Kh^{i,j}(L)$ are known as \textit{Khovanov homology groups} of $L$, and the indexes $i$ and $j$ as \emph{homological} and \emph{quantum gradings}, respectively. 

In \cite{LSoriginal} Lipshitz and Sarkar refined this invariant to obtain, for each quantum grading $j$, a spectrum $\cX_D^j$ whose stable homotopy type is a link invariant and whose cohomology is isomorphic to the Khovanov homology of the link, i.e., $H^{*}(\cX_D^j)\cong \Kh^{*,j}(L)$. Together with Lawson \cite{LLS15} they gave a neat construction of this spectrum as the realization of a $2$-functor $F^j$ from the cube category to the Burnside $2$-category (see also \cite{LLS2017cubo,LS-refinements}). 

Shortly after, another geometric realization of the (minimal) extreme Khovanov homology was given \cite{GMS} in terms of an independence simplicial complex constructed from the link diagram. This construction was later shown to coincide with the one of Lipshitz and Sarkar \cite{CS} as follows: functors from the cube to the Burnside $2$-category can be understood as generalizations of simplicial complexes; such a functor is a simplicial complex if it factors through the category of sets and takes values on singletons and the empty set. 

On the other hand, the Khovanov spectrum of a link in its maximal extreme quantum grading is Spanier-Whitehead dual to that of its mirror image in minimal grading. This construction was exploded in \cite{JACO} to make explicit computations of the Khovanov spectrum in (maximal) extreme quantum grading.

In 2018 the second author, together with Przytycki applied these techniques one step further, to the almost-extreme quantum grading. They restricted to the study of $1$-adequate link diagrams, i.e., those whose $1$-resolution contains no chords with both endpoints in the same circle (the extremal homology of these diagrams has been computed in \cite{Jozef-Rad, Rad2,Turaev1}), and built a pointed semi-simplicial set whose homology is isomorphic to the almost-extreme Khovanov homology of the diagram:
\begin{theorem}[\cite{PS}] If $D$ is a $1$-adequate link diagram, then
\begin{itemize}
\item[(i)] The functor $F_D^{\jalmax}$ gives rise to a pointed semi-simplicial set.
\item[(ii)] The realization of $F_D^{\jalmax}$ has the homotopy type of a wedge of spheres and possibly a copy of a (de)suspension of $\bR P^2$.
\end{itemize}
\end{theorem}
Their construction relates to the functor of Lawson, Lipshitz and Sarkar as follows: the category of pointed sets includes into the Burnside $2$-category, and if a functor from the cube to the Burnside $2$-category factors through it, then it gives rise to a pointed (augmented) semi-simpicial set.

Simplicial complexes and pointed semi-simplicial sets are simpler objects than functors to the Burnside $2$-category, to which many classical tools can be applied, so these results prompt the question
\begin{quote}
\emph{Are there simpler models of the spectra of Lipshitz and Sarkar in the almost-extreme quantum grading?}
\end{quote}
Unfortunately, it turns out that $F^\jalmax_D$ gives rise to a pointed semi-simplicial set if and only if $D$ is $1$-adequate. However, in this paper we overcome this problem by introducing a new functor $M_D$, whose realization coincides with that of $F^\jalmax_D$,
and gives rise to an (augmented) pointed semi-simplicial set in a broader number of cases. This is achieved by constructing a natural transformation between $M_D$ and $F^{\jalmax}_D$ and proving the following result:
\begin{atheorem}\label{thm:a} The natural transformation is a homology isomorphism, and therefore the realizations of $M_D$ and $F^{\jalmax}_D$ are homotopy equivalent.
\end{atheorem}
The functor $M_D$ can be understood as a ``change of basis'' of $F_D^{\jalmax}$: The relevant generators of the Khovanov chain complex in its almost-extreme degree (which is the same as the chain complex of $F^\jalmax_D$) are indexed by the circles of the state resolutions of $D$, whereas the generators of the chain complex associated to $M_D$ are indexed by the edges and connected components of certain graph whose vertices are the circles of the state resolutions. A crucial property at the almost-extreme quantum grading is that the relevant involved graphs are forests, and therefore both basis have the same size (as it should be).

The functor $M_D$ is simpler than $F^j_D$ in the sense that it takes values in morphisms sending generators to (multiples of) generators, whereas $F^j_D$ takes values in morphisms sending generators to sums of possibly different generators. 

This simplified new basis allows us to decompose $M_D$ in terms of independence simplicial complexes of simpler link diagrams (Proposition \ref{propcofibreseq} and Remark \ref{remarkMayerVietoris}). The last part of the paper is devoted to make explicit computations using this decomposition.

One of the advantages of $M_D$ over $F^\jalmax_D$ is that it gives rise to an augmented pointed semi-simplicial set if and only if the $1$-resolution of the diagram contains no alternating pairs (i.e., two chords whose endpoints alternate along the same circle). Moreover, the forementioned decomposition allows us to compute the homotopy type of the realization of $M_D$ in these cases:
\begin{atheorem}\label{thm:b} Let $D$ be a diagram whose $1$-resolution contains no alternating pairs.
\begin{itemize}
\item[(i)] The functor $M_D$ gives rise to an augmented pointed semi-simplicial set. 
\item[(ii)] If $D$ is not $1$-adequate then the realization of $M_D$ is homotopy equivalent to a wedge of spheres.
\end{itemize}
\end{atheorem}

In \cite{JACO} the second author conjectured (and proved for several cases) that the extreme Khovanov spectrum is always a wedge of spheres. If this conjecture were true, then the above decomposition would give an upper bound for the cone length of the spectrum $\cX^\jalmax_D$ (Remark \ref{remark:cone}). This bound seems to be far from optimal, as the computations of this paper and \cite{PS} produce spectra of cone length at most $2$, so we rise the following question:
\begin{quotation}
\emph{Are there diagrams whose almost-extreme Khovanov spectrum is not homotopy equivalent to a wedge of spheres and possibly (de)suspensions of projective planes?} 
\end{quotation}

The structure of the paper is as follows. In Sections 2 and 3 we recall the categorical notions that will be used along the paper. In Section 4 we develop some results on link diagrams which allow, in Section 5, to introduce the functor $M_D$ and prove Theorem \ref{thm:a}. In Section 6 we prove the first part of Theorem \ref{thm:b}. In Section 7 we decompose $M_D$ in terms of independence simplicial complexes and present three skein sequences which allow us to determine, in Section 8, the homotopy type of $M_D$ for all links with no alternating pairs, proving the second part of Theorem \ref{thm:b}. 

\subsection*{Acknowledgments} The authors thank Carles Casacuberta for his guidance and advice at early stages of this project.

\section{The Burnside 2-category}

A $2$-category is a higher category that generalizes the notion of category, in the sense that on top the objects and morphisms there are also $2$-morphisms. More precisely, a $2$-category consists of the data of
\begin{enumerate}
    \item a collection of objects,
    \item for each pair of objects $X,Y$, a category of morphisms $\hom(X,Y)$,
    \item for each object $X$, an object $\Id_X$ of $\hom(X,X)$ and
    \item for each triple of objects $X,Y,Z$, a composition functor $$\hom(Y,Z)\times \hom(X,Y)\to \hom(X,Z)$$
\end{enumerate}
which is associative and unital. The objects in $\hom(X,Y)$ are called \emph{$1$-morphisms} and the morphisms in $\hom(X,Y)$ are called \emph{$2$-morphisms}.

Recall that given two finite sets $X,Y$, a \emph{span} from $X$ to $Y$ is
a triple $(Q,\source,\target)$ where $Q$ is a finite set, and $\source$
and $\target$ are a pair of functions
$X\overset{\source}{\lla} Q\overset{\target}{\lra} Y$.

We say that a span is \emph{free} (resp. \emph{very free}) if the source map $s$ is injective (resp. bijective).

Given two spans  $X\overset{\source}{\longleftarrow} Q\overset{\target}{\longrightarrow} Y$ and $X\overset{\source'}{\longleftarrow} Q'\overset{\target'}{\longrightarrow} Y$, a \emph{fibrewise bijection} between them is a bijection $\tau\colon Q\to Q'$ such that $\source'\circ\tau=\source$ and $\target'\circ\tau=\target$:
\[\xymatrixcolsep{32pt}\xymatrixrowsep{10pt}\xymatrix{
    &Q\ar[ld]_{\source}\ar[rd]^\target\ar[dd]^{\tau}&\\
    X && Y \\
    &Q'\ar[lu]^{\source'}\ar[ru]_{\target'}&
}\]

The composition of two fibrewise bijections of spans is the composition of bijections,
and the identity morphism of a locally finite span is the identity bijection. This defines a \textit{category of spans} from the set $X$ to the set $Y$, which we note by $\cB(X,Y)$.

\begin{df}The Burnside $2$-category, denoted by $\cB$, is the $2$-category whose objects are finite sets, and the category of morphisms from a set $X$ to a set $Y$ is given by $\cB(X,Y)$, the category of spans from $X$ to $Y$; in other words, the $1$-morphisms are given by spans and $2$-morphisms are fibrewise bijections.
\end{df}

Let $X_1, X_2$ and $X_3$ be three sets in $\cB$. Given a span $(Q_1,\source_1,\target_1)$ in $\cB(X_1,X_2)$ and a span $(Q_2,\source_2,\target_2)$ in $\cB(X_2,X_3)$, their composition is given by their fibre product $(Q_1 \times_{X_2} Q_2, s,t)$: \[\xymatrixcolsep{16pt}\xymatrixrowsep{6pt}\xymatrix{
    && Q_1\times_{X_2} Q_2\ar@{-->}[dl]  \ar@{-->}[dr] \ar@/_1.5pc/[ddll]_{\source}  \ar@/^1.5pc/[ddrr]^{\target} && \\
    &Q_1\ar[dl]^{\source_1}\ar[dr]_{\target_1} && Q_2\ar[dl]^{\source_2}\ar[dr]_{\target_2} & \\
    X_1 && X_2 && X_3. }\]

The composition of two fibrewise bijections $\tau_1\colon Q_1 \to
Q_1'$ and $\tau_2 \colon Q_2 \to Q_2'$ is also given by their
fibre product $\tau_1 \times \tau_2 \colon Q_1 \times_{X_2} Q_2
\to Q'_1 \times_{X_2} Q'_2$ as follows

\[\xymatrixcolsep{24pt}\xymatrixrowsep{8pt}\xymatrix{
    &Q_1\ar[dl]\ar[dr]\ar[dd]^{\tau_1} && Q_2\ar[dl]\ar[dr]\ar[dd]^{\tau_2} & \\
    X_1 && X_2 && X_3 \\
    &Q'_1\ar[ul]\ar[ur] && Q'_2\ar[ul]\ar[ur] & \\
}\qquad
\xymatrix{
    & Q_1\times_{X_2} Q_2\ar[dd]^{\tau_1\times\tau_2}\ar[dr]\ar[dl] & \\
    X& & Y,\\
    &Q'_1\times_{X_2} Q'_2,\ar[ur]\ar[ul]&
}\]
and the identity on a set $X$ is the span $X\overset{\Id}{\la} X\overset{\Id}{\to} X$.

\begin{remark}
Note that a span $\s = (Q,\source,\target)$ from $X$ to $Y$ is determined by the matrix whose entries are the sets $\s_{x,y} := \source^{-1}(x)\cap \target^{-1}(y)$, for every $x\in X$, $y \in Y$. Giving a fibrewise bijection $\tau$ from a span $\s$ to a span $\s'$ is the same as giving, for each $x\in X$ and each $y\in Y$, a bijection $\tau_{x,y}$ from $\s_{x,y}$ to $\s'_{x,y}$. The composition of $\s\colon X\to Y$ and $\s'\colon Y\to Z$ is the span $\s''$ with $\s''_{x,z} = \bigcup_{y\in Y}\s_{x,y}\times \s'_{y,z}$. The span $\s$ is determined  by the formal sum
\[\s(x) = \sum_{y\in Y} \s_{x,y}\cdot y,\] for each $x \in X$. If a coefficient $\s_{x,y}$ is not specified, we understand that $\s_{x,y} = \{x\}$ (for example, if $\s(x) = y+z$, then $\s(x) = \{x\} \cdot y + \{x\} \cdot z$). 
\end{remark}

The category of finite sets $\Set$ can be mapped into the category of pointed finite sets $\Setp$ by sending $X$ to $X_+:=X\cup \{\boldast\}$, and a morphism $f\colon X\to Y$ to the morphism $f_+$ that coincides with $f$ on $X$ and sends the basepoint of $X_+$ to the basepoint of $Y_+$. Moreover, the category $\Setp$ sits inside $\cB$ by sending a pointed finite set $X$ to $X\smallsetminus \{\boldast\}$, and a morphism $f\colon X\to Y$ to the span
\[X\smallsetminus \{\boldast\}\hookleftarrow X\smallsetminus f^{-1}(\boldast)\overset{f}{\to} Y\smallsetminus \{\boldast\}.\]

The inclusions $\Set \, \hookrightarrow \, \Setp \, \hookrightarrow \, \cB$ induce equivalences between the following $2$-subcategories of $\cB$:
\begin{itemize}
\item The essential image of $\Setp$ in $\cB$ is the $2$-subcategory of free spans.
\item The essential image of $\Set$ in $\cB$ is the $2$-subcategory of very free spans.
\end{itemize}

Fix now a commutative ring $R$ with unit, and let $\Mod{R}$ be the category of modules over $R$. There is a functor \begin{equation}\label{functorA}
\btoab \colon \cB\to \Mod{R}
\end{equation}
sending the finite set $X$ to the free $R$-module $R\langle X\rangle$, and a span $\s\colon X\to Y$ to the homomorphism $f\colon R\langle X\rangle\to R\langle Y\rangle$ given, for every $x\in X$, by
\[f(x) = \sum_{y\in Y} |\s_{x,y}|\cdot y.\]
Note that every pair of spans connected by a $2$-morphism are sent to the same homomorphism, and therefore the functor $\btoab$ is well-defined.

\section{Cubes of pointed sets and augmented semi-simplicial pointed sets}

The \emph{$n$-dimensional cube} $\cube{n}$ is the partially ordered set (poset) whose elements are $n$-tuples $u=(u_1,\ldots, u_n) \in \{0,1\}^n$ endowed with the standard partial order so that $u\geq v$ if $u_i \geq v_i$ for all $i$.  Following \cite{LLS15}, we regard $\cube{n}$ as the category whose objects are the elements of this poset and, for any two such elements $u,v$, the morphism set $\operatorname{Hom}(u,v)$ has a single element $\varphi_{u,v}$ if $u\geq v$, and is empty otherwise. This category has an initial element $\uno=(1,1,\ldots,1)$ and a terminal element $\cero = (0,0,\ldots,0)$.

The poset map $|\cdot|\colon \cube{n}\to \bN$ given by $|u|=\sum_{i} u_i$ assigns a grading to each vertex in the cube.
We write $u\succ v$ if $u>v$ and $|u-v|=1$, and we write $u\succ_i v$ if, additionally, $u$ and $v$ differ in the $i$th-coordinate, i.e., if $u_i=1$, $v_i=0$ and $u_j = v_j$ for $i\neq j$. 

Given $n \in \mathbb{Z}$, $n\geq -1$, the finite ordinal $[n]$ is the linearly ordered set $\{0<1<\ldots<n\}$. The \textit{augmented semi-simplicial category}, $\Deltainjaug$, has as objects these finite ordinals, and as morphisms injective order-preserving maps between them. The inclusion $\partial_i \colon [n-1]\hookrightarrow [n]$ that forgets the $i$-th element is called $i$th-face map. 
\begin{df}
Given a category $\mathcal{C}$, let $\mathcal{C}^\op$ denote its opposite category, obtained by reversing the morphisms. An \textit{$n$-dimensional cube of pointed (finite) sets} is a functor $F\colon \cube{n} \to \Setp$, while an \textit{augmented semi-simplicial pointed (finite) set} is a functor $X\colon \Deltainjaug^\op\to \Setp$. As usual, we write $X_n$ and $\partial_i$ for $X([n])$ and $X(\partial_i)$, respectively.
\end{df}

These two categories are related by a  functor $\lambda\colon \cube{n}\to \Deltainjaug^\op$ which maps every vertex of the cube $u$ to the ordinal $[|u|-1]$, and every morphism $u\succ_i v$ to the opposite of the $\left(\sum_{j<i} u_i\right)$th-face map.

Taking left Kan extension along the functor $\lambda\colon \cube{n}\to \Deltainjaug^\op$ defines a functor
\begin{equation}\label{Lambda}
\Lambda\colon \Setp^{\cube{n}}\lra \presh{\Deltainjaug^\op}{\Setp},
\end{equation} 
whose value on a cube of pointed sets $F$ is defined explicitely as follows: its set of $k$-simplices is $\Lambda(F)_k = \coprod_{|u|=k+1} F(u)$, for every $u\in \cube{n}$, and the $i$th-face map $\partial_i\colon \Lambda(F)_k\to \Lambda(F)_{k-1}$ is the union $\coprod_{|u|=k+1} F(u>u[i])$, where $u[i]$ is the vertex of the cube $\cube{n}$ obtained by replacing the $(i+1)$th one in $u$ by a zero.

Recall from \cite[Definition 4.1]{LLS2017cubo} that a \textit{strictly unitary lax 2-functor} $F$ from the cube $\cube{n}$ to the Burnside category $\mathcal{B}$ consists of the data:
\begin{enumerate}
\item a finite set $F(v) \in \operatorname{Ob}(\mathcal{B})$, for each $v \in \{0,1\}^n$.
\item a finite span $F(\varphi_{u,v}): F(u) \to F(v)$ in $\mathcal{B}$, for each morphism $\varphi_{u,v}$ in $\cube{n}$.
\item a $2$-isomorphism $F_{u, v, w}$ from $F(\varphi_{v,w})\circ F(\varphi_{u,v})$ to $F(\varphi_{u,w})$, for each decomposition $\varphi_{u,w} = \varphi_{v,w} \circ \varphi_{u,v}$;
\end{enumerate}
satisfying that, for every $u > v > w > z$, the following diagram commutes:
  \[
  \xymatrix@C=10ex{
    F(\varphi_{w,z})\times_{F(w)} F(\varphi_{v,w})\times_{F(v)}F(\varphi_{u,v})\ar[r]^-{\, \operatorname{Id} \times F_{u, v,w} \,} \ar[d]^-{F_{v,w,z} \times \operatorname{Id}}&
    F(\varphi_{w,z}) \times_{F(w)} F(\varphi_{u,w}) \ar[d]_-{F_{u,w,z}} \\
    F(\varphi_{v,z})\times_{F(v)} F(\varphi_{u,v}) \ar[r]^-{F_{u,v,z}} &F(\varphi_{u,z}).\\
  }
  \]

%

In practice, when defining a strictly unitary lax 2-functor, we use the following result, 
from \cite[Section 4]{LLS2017cubo}. We keep the notation $\xymatrix@C=1.5ex{u\ar@{-{*}}[r]&v}$ for the morphism $u \succ v$:

\begin{lemma}\label{lematransformation} A strictly unitary lax $2$-functor from the cube category $\cube{n}$ to the Burnside category $\cB$ is uniquely determined (up to natural isomorphism) by the following data:
\begin{enumerate}[label=(D\arabic*),ref=(D\arabic*)]
\item\label{data:set} for each vertex $v \in \{0,1\}^n$, a finite set $F(v) \in \operatorname{Ob}(\mathcal{B})$;
\item\label{data:correspondence} for each $u\succ v$, a finite span $F(\varphi_{u,v}): F(u) \to F(v)$ in $\mathcal{B}$;
\item\label{data:bijection} for each two-dimensional face 
$\vcenter{\xymatrix@R=0ex@C=1ex{&v\ar@{-{*}}[dr]\\u\ar@{-{*}}[ur]\ar@{-{*}}[dr]&&w\\&v'\ar@{-{*}}[ur]}}$
of the cube
, a $2$-morphism \[F_{u,v,v',w}\colon
  F(\varphi_{v,w})\times_{F(v)} F(\varphi_{u,v})\to
  F(\varphi_{v',w})\times_{F(v')} F(\varphi_{u,v'}),\]
\end{enumerate}
satisfying the following two conditions:
\begin{enumerate}[label=(C\arabic*),ref=(C\arabic*)]
\item\label{condition:matching} for every two-dimensional face
	$\vcenter{\xymatrix@R=0ex@C=1ex{&v\ar@{-{*}}[dr]\\u\ar@{-{*}}[ur]\ar@{-{*}}[dr]&&w\\&v'\ar@{-{*}}[ur]}}$,
  $F_{u,v',v,w}=F_{u,v,v',w}^{-1},$
\item\label{condition:hexagon} for every three-dimensional face
$\vcenter{\xymatrix@R=0.8ex@C=1.5ex{&v\ar@{-{*}}[r]\ar@{-{*}}[dr]&w''\ar@{-{*}}[dr]\\u\ar@{-{*}}[ur]\ar@{-{*}}[r]\ar@{-{*}}[dr]&v'\ar@{-{*}}[ur]\ar@{-{*}}[dr]&w'\ar@{-{*}}[r]&z\\&v''\ar@{-{*}}[r]\ar@{-{*}}[ur]&w\ar@{-{*}}[ur]}}$,
  the following commutes:
  \[
  \xymatrix@C=10ex{
    F(\varphi_{w'',z})\times_{F(w'')}F(\varphi_{v,w''})\times_{F(v)}F(\varphi_{u,v})\ar[r]^-{F_{v,w'',w',z}\times\Id}\ar[d]^-{\Id\times
      F_{u,v,v',w''}}&
    F(\varphi_{w',z})\times_{F(w')}F(\varphi_{v,w'})\times_{F(v)}F(\varphi_{u,v})\ar[d]_-{\Id\times
      F_{u,v,v'',w'}} \\
    F(\varphi_{w'',z})\times_{F(w'')}F(\varphi_{v',w''})\times_{F(v')}F(\varphi_{u,v'})\ar[d]^-{F_{v',w'',w,z}\times\Id}&F(\varphi_{w',z})\times_{F(w')}F(\varphi_{v'',w'})\times_{F(v'')}F(\varphi_{u,v''})\ar[d]_-{F_{v'',w',w,z}\times\Id}\\
    F(\varphi_{w,z})\times_{F(w)}F(\varphi_{v',w})\times_{F(v')}F(\varphi_{u,v'})\ar[r]^-{\Id\times
      F_{u,v',v'',w}}&F(\varphi_{w,z})\times_{F(w)}F(\varphi_{v'',w})\times_{F(v'')}F(\varphi_{u,v''}).
  }
  \]
\end{enumerate}
\end{lemma}

\begin{df}\cite{LLS2017cubo}\label{df:laxcube} A \emph{natural transformation} $\alpha \colon F\to G$ between two strictly unitary lax $2$-functors $F,G\colon \cube{n} \to \cB$ is a strictly unitary lax $2$-functor $H\colon \cube{n+1}\to \cB$ such that $H|_{\cube{n}\times \{1\}} = F$ and $H|_{\cube{n}\times \{0\}} = G$. For every $u\in \cube{n}$, we denote by $\alpha_u$ the morphism $H(\varphi_{(u,1),(u,0)})$. For every $u \succ v \in \cube{n}$, we denote $\alpha_{u,v}$ the $2$-morphism $H_{(u,1),(u,0),(v,1), (v,0)}$.
\end{df}

We write $\presh{\cD}{\cC}$ for the category whose objects are strictly unitary lax $2$-functors from $\cD$ to $\cC$ and whose morphisms are natural transformations between them. When $\cC = \cB$ and $\cD = \cube{n}$ the objects are called \emph{Burnside cubes}.  

\subsection{Realizations and totalizations}
Let $\cC$ be a model category, as the category of (pointed) topological spaces $\Top$ (resp. $\Topp$), the category of spectra $\Sp$, or the category $\Ch(R)$ of chain complexes of $R$-modules, with $R$ a commutative ring.

The \emph{totalization} of a functor $F\colon \cube{n}\to \cC$ is the object of $\cC$ given by
\[\Tot F = \hocofib\left(\hocolim \left(F|_{\cube{n}\smallsetminus\{0\}}\right)\to F(0)\right).\]
This defines a functor
\[\Tot\colon \cC^{\cube{n}}\lra \cC.\]

On the other hand, the \emph{relative realization} of a functor $X\colon \Deltainjaug^\op\to \cC$ is the object in $\cC$ given by
\[\|X\| = \hocofib\left(\hocolim \left(X|_{\Deltainjaug\smallsetminus\{[-1]\}}\right)\to X_{-1}\right),\]
which defines a functor
\[\|-\|\colon \presh{\Deltainjaug^\op}{\cC}\lra \cC.\]
A semi-simplicial object $X\colon \Deltainj^\op\to \cC$ may be seen as an augmented semi-simplicial set by defining $X_{-1}=*$, the final object of $\cC$. Its relative realization is the suspension of the realization of $X$:
\[\|X\| \cong  \Sigma|X|.\]

Recall now the functor $\lambda\colon \cube{n}\to \Deltainjaug^\op$ mapping every $u\in \cube{n}$ to $[|u|-1]$, and every morphism $u\succ_i v$ to the opposite of $\left(\sum_{j<i} u_i\right)th$-face map. Since $\lambda$ is cofinal, $\lambda(\cube{n}\smallsetminus\{0\}) = \Deltainjaug\smallsetminus\{[-1]\}$ and $\lambda(\cero) = [-1]$, we have that
\begin{equation}\label{eq:19} \|\Lambda(F)\|\simeq \Tot F.\end{equation}

\subsection{From Burnside cubes to chain complexes}\label{ss:32} Let $R$ be a ring, and let $\Ch(R)$ be the category of chain complexes of $R$-modules. There is a functor
\[K \colon \cB\lra \Mod{R} \lra \Ch(R),\]
where the first functor is $\btoab$, defined in \eqref{functorA},  and the second functor sends an $R$-module to the chain complex given by the $R$-module concentrated at degree $0$.

We define the functor $C_*(-;R)$ as the composition
\[C_*(-;R)\colon \cB^\cube{n}\lra \Ch(R)^\cube{n}\overset{\Tot}{\lra} \Ch(R),\]
where the first functor is the result of composing a $2$-functor $F\colon \cube{n} \to \cB$ from $\cB^\cube{n}$ with $K$.

Its value at the Burnside cube $F$ can also be described as the chain complex whose $k$-cochains are
\[C_k(F;R) = \bigoplus_{u\in \cube{n}, |u|=k} R\langle F(u)\rangle\]
and whose differential is
\[\partial = \sum_{|u|=k, |v|=k-1} (-1)^{\sum_{j<i}u_i}\btoab \circ F(u\succ_i v).\]

Observe that since $\cB\to \Mod{R}$ sends $2$-morphisms to identities, a natural transformation between Burnside cubes induces a morphism of chain complexes.

\subsection{From Burnside cubes to spectra} In \cite{LLS15}, Lawson, Lipshitz and Sarkar gave an explicit construction of the totalization functor for cubes in the Burnside category. They associated, to each cube $F\colon\cube{n}\to \cB$, a spectrum $\Tot F$, and to each map $f \colon F\to G$ of Burnside cubes, a map $\Tot f \colon \Tot F\to \Tot G$, both well-defined up to homotopy. Additionally, the homology $H_*(\Tot F;R)$ was isomorphic to the homology of $C_*(F;R)$, and the map induced by $\Tot f$ on homology was the one induced by $C_*(f;R)$.

\begin{remark}\label{remark:spectraequiv}
If $C_*(f;R)$ is a homology isomorphism then $\Tot f$ is a homotopy equivalence of spectra.   
\end{remark}

On the other hand, in \cite{CS} the authors showed that if a cube $F\colon\cube{n}\to \cB$ factors through some functor $\tilde{F}\colon \cube{n}\to \Setp$, then $\Tot F = \Sigma^\infty \Tot \tilde{F}$, where the homotopy colimit is taken in pointed topological spaces. Therefore, by \eqref{eq:19}, we have:

\begin{proposition}\label{prop:setp} If a Burnside cube $F\colon \cube{n}\to \cB$ in the Burnside category factors through $\Setp$, then the pointed semi-simplicial set $\Lambda(\tilde{F})$ satisfies  $$\Tot F \simeq \Sigma^\infty\displaystyle\|\Lambda(\tilde{F})\|.$$
\end{proposition}

\begin{remark}
The spectrum $\Tot F$ is denoted $|F|$ in \cite{LLS15}. The chain complex $C_*(F;R)$ is denoted $\Tot F$ in  \cite{LLS2017cubo}. We reserve the notation $\Tot$ for the totalization of a cube, $\|\cdot\|$ for the relative realization of an augmented semi-simplicial object and $| \cdot |$ for the realization of a semi-simplicial object.  
\end{remark}

\section{Knots and graphs}

\subsection{States}

Let $D$ be an oriented link diagram with $n$ ordered crossings $\{c_1 < c_2 < \ldots <c_n\}$, where $n_+$ ($n_-$) of them are positive (negative). A \emph{(Kauffman) state} of the diagram $D$ is an assignation of a label, $0$ or $1$, to each crossing in $D$. The order of the crossings induces a bijection between the set of states of $D$ and the elements of $\cube{n}$ by considering $u \in \cube{n}$ as the state that assigns the label $u_i$ to $c_i$.

Smoothing a crossing $c_i$ of $D$ consists on replacing it by a pair of arcs and a segment connecting them; the way we smooth the crossings depends on its associated label $u_i$, as shown in Figure \ref{Labels}. The result of smoothing each crossing of $D$ according to its label is the chord diagram $D(u)$, a collection of disjoint circles together with some segments that we call $0$- and $1$-chords, depending on the value of the associated $u_i$. We represent $1$-chords as light segments, and $0$-chords as dark ones.

\begin{figure}[h]
\centering
\includegraphics[width = 8.15cm]{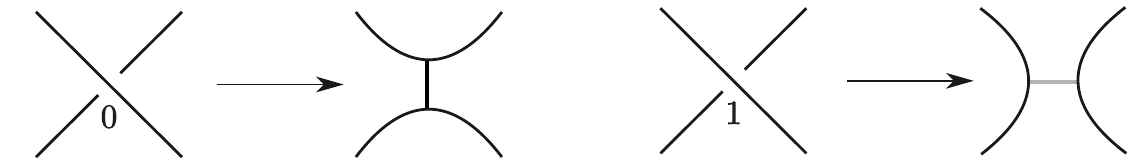}
\caption{\small{The smoothing of a crossing according to its $0$ or $1$ label.}}
\label{Labels}
\end{figure}

We introduce the following notation to refer to the circles and chords of $D(u)$:
\begin{itemize}
\item $Z(u)$ denotes the set of circles of $D(u)$,
\item $E(u)$ denotes the set of $0$-chords of $D(u)$,
\item $\bar{E}(u)$ denotes the set of $1$-chords of $D(u)$,
\item $e_i(u)$ denotes the $i$th chord in $D(u)$,
\item $\cO_i(u)$ denotes the set of circles containing an endpoint of $e_i(u)$ (thus $|\cO_i(u)|$ is either equal to $1$ or $2$). If $|\cO_i(u)|=2$ (resp. $|\cO_i(u)|=1$), we say that $e_i(u)$ is a \textit{bichord} (resp. \textit{monochord}).  
\end{itemize}

Given a state $u$ of $D$, we define its associated \textit{state graph}, $\G(u)$, as the labelled graph obtained by collapsing each circle of $D(u)$ to a vertex so that each chord in $D(u)$ becomes an edge in $\G(u)$; each edge inherits a label, $0$ or $1$, from the associated chord. The circles and chords of $D(u)$ are in bijection with the vertices and edges of $\G(u)$. Therefore, the above notation introduced to refer to the circles and chords of $D(u)$ will be used to refer to the vertices and edges of $\G(u)$. In particular, the set of vertices of $\G(u)$ is $Z(u)$ and the set of edges labeled by $0$ and $1$ are $E(u)$ and $\bar{E}(u)$, respectively.

In this setting, write $G(u)$ for the subgraph obtained after removing the $1$-edges from $\G(u)$. See Figure \ref{Example_graph} for such an example. We consider loops as length-1 cycles.

\begin{figure}
\centering
\includegraphics[width = 10.5cm]{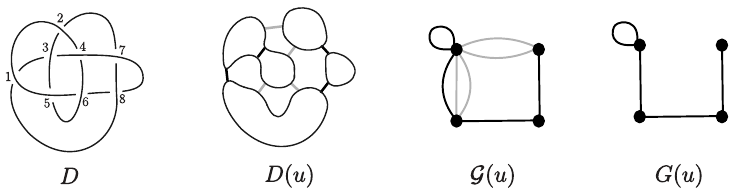}
\caption{\small{A link diagram $D$, the configuration $D(u)$ corresponding to the state $u= (0,1,0,1,1,1,0,0)$ and the associated graphs $\mathcal{G}(u)$ and $G(u)$.}}
\label{Example_graph}
\end{figure}

Recall that given $u,v \in \cube{n}$, we write $u \succ_i v$ if both vectors are equal in all but the $i^{th}$ coordinate, where $u_i = 1$ and $v_i=0$. The partial order $>$ is defined as the transitive closure of the above relation.

If $u\succ_i v$, then $D(v)$ is obtained from $D(u)$ by doing surgery on $D(u)$ along the $1$-chord $e_i(u)$ and adding a new chord $e_i(v)$.  Observe that, depending on the cardinality of $\cO_i(u)$, there are two possible surgeries:

\begin{figure}
\centering
\includegraphics[width = 12cm]{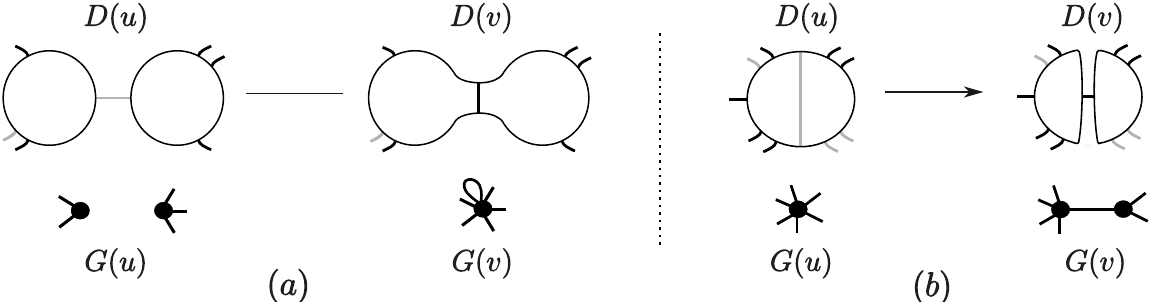}
\caption{\small{A merging and a splitting are illustrated in $(a)$ and $(b)$, respectively. The endpoints of some chords have been drawn in order to show the effect of both surgeries in the graphs associated to the states.}}
\label{Surgery}
\end{figure}

\begin{itemize}
\item If $|\cO_i(u)| = 2$, then $|\cO_i(v)| = 1$ and $D(v)$ is obtained from $D(u)$ by joining two circles into one. The graph $G(v)$ is obtained from $G(u)$ by identifying the two vertices in $\cO_i(u)$ and adding a loop (the edge $e_i(v)$) based on the identified vertex, as shown in Figure \ref{Surgery}$(a)$. We say that $u\succ_i v$ a \textit{merging}.
\item If $|\cO_i(u)| = 1$, then $|\cO_i(v)| = 2$ and $D(v)$ is obtained from $D(u)$ by splitting one circle into two. In this case, $G(u)$ is obtained from $G(v)$ by collapsing the edge $e_i(v)$,  as shown in Figure \ref{Surgery}$(b)$. We say that $u\succ_i v$ is a \textit{splitting}.
\end{itemize}

\begin{remark}
Given two states $u$ and $v$ of $D$ so that $u>v$, with $|u|-|v|=k$, there exist $k!$ possible chains
$$u = u_0 \succ_{i_1} u_1 \succ_{i_2}\ldots\succ_{i_{k}} u_k = v$$
connecting them. However, the total composition of the one-by-one surgeries induced by each of the possible chains does not depend on the chosen chain.
\end{remark}



Given a surgery $u \succ_i v$, we compare $G(u)$ and $G(v)$:
\begin{align*}
|E(v)| &= |E(u)|+1. \\
|Z(v)| &= \begin{cases} |Z(u)|+1 & \text{if $u\succ_i v$ is a splitting,} \\ |Z(u)|-1 & \text{if $u\succ_i v$ is a merging.}\end{cases} \\
|\pi_0 G(v)| &= \begin{cases}
|\pi_0 G(u)| & \text{if $u\succ_i v$ is a splitting, } \\
|\pi_0 G(u)| & \text{if $u\succ_i v$ is a merging and $\cO_i(u)\to \pi_0 G(u)$ is not injective,} \\
|\pi_0 G(u)|-1 & \text{if $u\succ_i v$ is a merging and $\cO_i(u)\to \pi_0 G(u)$ is injective.}
\end{cases} 
\end{align*}

\begin{df} Given a state $u \in \cube{n}$ with $|u| = n-k$ and a chain \begin{equation}\label{chain0tou}\uno = u_0 \succ_{i_1} u_1 \succ_{i_2} \ldots \succ_{i_{k}} u_k = u \end{equation} connecting $u$ to the state $\uno$, we define $\Phi(u)$ as the number of loops in $\{e_{i_{j}}(u_{j})\}_{j=1}^{k}$. In other words, $\Phi(u)$ counts the number of mergings in the chain.
\end{df}

Equalities above allow us to describe $G(u)$ in terms of $G(\uno)$ and $\Phi(u)$:
\begin{align}
\label{eq:4}
|Z(u)| &= |Z(\uno)| + n - |u| - 2\Phi(u), \\ \label{eq:4.2}
\chi(G(u)) &= \chi(G(\uno)) - 2\Phi(u),  \\ \label{eq:4.3}
\flat{\uno} - \Phi(u) &\leq \flat{u} \leq \flat{\uno}
\end{align}

In particular, the above relations imply that $\Phi(u)$ does not depend on the chosen chain connecting the states $\uno$ and $u$.

\begin{lemma}\label{alternatingcycles}
Let $u \in \cube{n}$ be a state of a link diagram $D$.
\begin{enumerate}
\item If $\Phi(u)=0$, then $G(u)$ is a forest, namely, it consists of a collection of $|Z(\uno)|$ contractible components.
\item If $\Phi(u)=1$, then the number of connected components of $G(u)$, $|\pi_0 G(u)|$, is either $|Z(\uno)|$ or $|Z(\uno)|-1$. In addition:
\begin{enumerate}
\item If $|\pi_0 G(u)| = |Z(\uno)|-1$, then there is one single cycle in $G(u)$,
\item If $|\pi_0G(u)| = |Z(\uno)|$, then $G(u)$ contains exactly two cycles $\mathfrak{c}_1$ and $\mathfrak{c}_2$ sharing a common vertex $z$. Moreover, in $D(u)$ the chords corresponding to the edges of $\mathfrak{c}_1$ adjacent to the vertex $z$ alternate  with those of $\mathfrak{c}_2$ along the circle $z$. Therefore the chords corresponding to the edges of $\mathfrak{c}_1$ and the chords corresponding to the edges of $\mathfrak{c}_2$ lie in different regions of $\bR^2\smallsetminus z$.
\end{enumerate}
\end{enumerate}
\end{lemma}

 \begin{proof}
 The assertions concerning the homotopy type of $G(u)$ follow from \eqref{eq:4} and \eqref{eq:4.2} after computing the Euler characteristic of $G(u)$ as
\[\flat{u}-\#(\text{cycles in } G(u)) = \chi (G(u)) = \chi(G(\uno)) - 2\Phi(u) = |Z(\uno)|-2\Phi(u).\]

If $\Phi(u)=0$, then $\flat{u} = \flat{\uno} = |Z(\uno)|$ by \eqref{eq:4.3}, and therefore $G(u)$ contains no cycles. In the case (2a), $\flat{u} = |Z(\uno)|-1$, and therefore $G(u)$ has a single cycle, whereas in case (2b), $\flat{u} = |Z(\uno)|$, and therefore there are two cycles in $G(u)$.

We prove now that when $\Phi(u) = 1$ and $\flat{u} = \flat{\uno}$, the chords of $\mathfrak{c}_1$ adjacent to the circle $z$ alternate  with those of $\mathfrak{c}_2$ along the boundary of $z$.

Consider a chain as \eqref{chain0tou}, and let $u_{r-1}\succ_{i_r} u_{r}$ be the states right before and right after the merging is performed, respectively. Write $\mathcal{O}_{i_r}(u_{r-1}) = \{z_1, z_2\}$ and define $\varepsilon$ as the (possibly empty) set containing those edges $e_{i_j}$ in the chain so that $\mathcal{O}_{i_j}(u_{r-1}) = \{z_1,z_2\}$, for $r < j \leq k$.

Since $\Phi(u_{r-1})=0$, we have just shown that $G(u_{r-1})$ consists on $|Z(\uno)|$ disjoint trees. When passing from $u_{r-1}$ to $u_{r}$, vertices $z_1$ and $z_2$ are identified and the $1$-edge $e_{i_r}$ becomes a $0$-loop (i.e., a length-one cycle) in $G(u_{r})$. Figure \ref{cycles}(a)-(b) illustrates this process.

\begin{figure}
\centering
\includegraphics[width = 12.7cm]{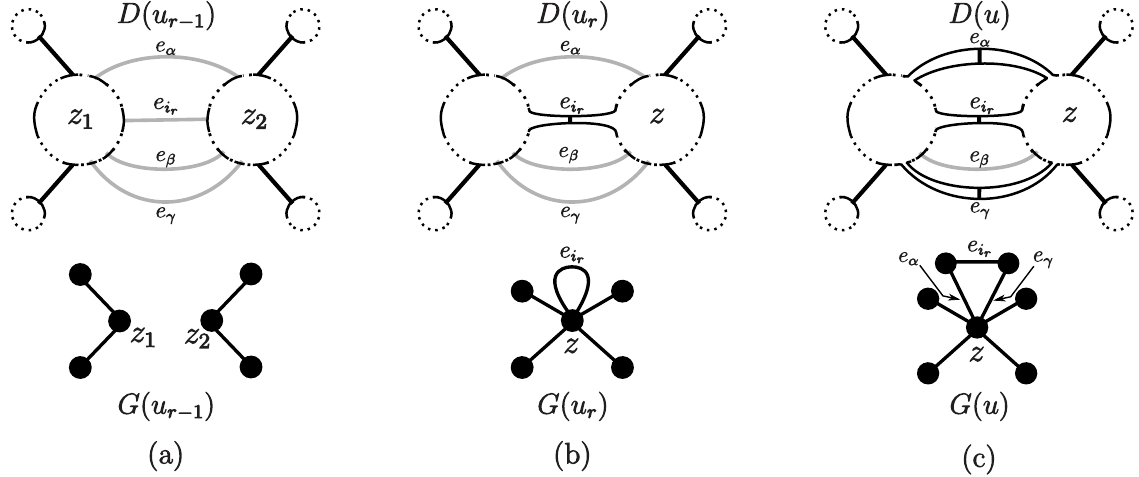}
\caption{\small{Diagrams $D(u_{r-1})$, $D(u_{r})$, $D(u)$ and the associated graphs illustrating the proof of Lemma \ref{alternatingcycles} are shown. Recall that dark (resp. light) chords corresponds to $0$-chords (resp. $1$-chords). In this case $\varepsilon=\{e_\alpha, e_\gamma\}$.}}
\label{cycles}
\end{figure}

The case $(2a)$ corresponds to the case when $z_1$ and $z_2$ belong to different connected components of $G(u_{r-1})$. Since no more mergings are possible, no more cycles are created. Moreover, the length of the (unique) cycle $\mathfrak{c}_1$ in $G(u)$ is $|\varepsilon|+1$ (see Figure \ref{cycles}(c)).

The case $(2b)$ corresponds to the case when $z_1$ and $z_2$ belong to the same connected component: the fact that there is a path connecting $z_1$ and $z_2$ in $G(u_{r-1})$ implies that an additional cycle $\mathfrak{c}_2$ is created when identifying both vertices. As illustrated in Figure \ref{cycles}(b)-(c), the condition on the alternacy of the chords holds. The planarity of $D(u)$ completes the proof.
\end{proof}


\subsection{Enhanced states}

An \emph{enhacement} of a state $u \in \cube{n}$ is a map $x$ assigning a label $+1$ or $-1$ to each of the circles in $Z(u)$; we note by $(u,x)$ the associated \emph{enhanced state}. Write $Z_+(u,x)$ and $Z_-(u,x)$ for the subsets of elements of $Z(u)$ labeled by $+1$ and $-1$, respectively. Define, for the enhanced state $(u,x)$, the integers
$$
h(u,x) =  -n_- + |u|, \quad \quad q(u,x) = n_+ - 2n_- + |u| + |Z_+(u,x)| -|Z_-(u,x)|,
$$
which are the \emph{homological} and \emph{quantum} gradings for Khovanov homology, respectively \cite{LSoriginal}.

Let $x_+$ be the constant enhacement with value $+1$ and, for a given circle $z \in Z(u)$, let $x_z^+$ be the enhacement assigning a positive label to every circle but $z$. Sometimes we will write $u_+=(u,x_+)$.

Define $$j_{\max}(D) = \max\{{q(u,x) \mid (u,x) \mbox{ is an enhanced state of } D}\}$$ and $\jalmax(D) = j_{\max}(D)-2$; we refer to these numbers as \textit{extreme} and \textit{almost-extreme (quantum) gradings} for Khovanov homology of the diagram $D$. It turns out that \begin{align*}
j_{\max}(D) &= q(\uno,x_+) & & \text{and}& \jalmax(D) &= q(\uno,x_z^+), \, \text{ for any } z \in Z(\uno).\end{align*}

\begin{proposition}\label{valueL}
Let $(u,x)$ be an enhanced state of a diagram $D$ with $q(u,x) = j$. Then,
$$|Z_-(u,x)| = \frac{j_{\max}(D)-j}{2} - \Phi(u).$$
\end{proposition}

\begin{proof}
From the definition of quantum grading, it holds that
\begin{align*} j_{\max}(D)-j & = n+ |Z(\uno)| -|u| - |Z_+(u,x)| + |Z_-(u,x)| \\ & = n + |Z(\uno)| - |u| - |Z(u)| + 2|Z_-(u,x)|. 
\end{align*}
Relation \eqref{eq:4} completes the proof.
\end{proof}

\begin{corollary}\label{corollaryL}
Let $(u,x)$ be an enhanced state of $D$ satisfying $q(u,x) = \jalmax(D)$. Then $\Phi(u)\leq 1$. Moreover,
\begin{enumerate}
\item if $\Phi(u)=0$, then $x = x_z^+$ for some $z \in Z(u)$.
\item if $\Phi(u)=1$, then $x = x_+$.
\end{enumerate}
\end{corollary}

Since we are interested in studying the almost-extreme Khovanov complex of a link diagram, in the next section we study the characterization of those states $u$ so that $\Phi(u)$ equals $0$ or $1$.

\subsection{States with $\Phi(u)=0,1$}

We are interested now in characterizing those states taking part in the almost-extreme Khovanov complex, i.e., given a state $u \in \cube{n}$ of $D$, we want to determine whether there exists an enhacement $x$ so that $q(u,x)=\jalmax(D)$, just by looking at $D(\uno)$.

\begin{df} Given $u\in \cube{n}$ a state of $D$ and $a,b,c,d$ chords in $D(u)$, we say that:
\begin{enumerate}
\item $a$ and $b$ are \textit{parallel} if they are bichords with their endpoints in the same circles, i.e., $\cO_a(u) = \cO_b(u)=\{z,z'\}$. 
\item $a,b$ form an \emph{alternating pair} if they are monochords and their endpoints alternate along the same circle, as illustrated in Figure~\ref{Altpairtriple}~$(a)$ (compare to ladybug configuration from \cite{LSoriginal}).
\item $a,b,c$ form an \emph{alternating triple} if $\cO_a(u) = \cO_b(u)=\{z,z'\}$ and their endpoints alternate with the endpoints of the monochord $c$ along $z$. See Figure~\ref{Altpairtriple}~$(b)$.
\item $a,b,c,d$ form a \emph{mixed alternating pair} if the only alternating pairs among them are $(a,b), (b,c)$ and $(c,d)$. See Figure \ref{Altpairtriple}~$(c)$.
\end{enumerate}
\end{df}

\begin{figure}
\centering
\includegraphics[width = 9.1cm]{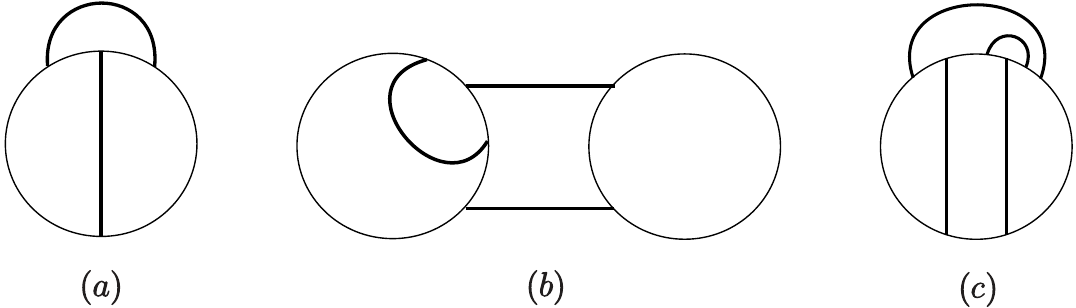}
\caption{\small{Chords forming an alternating pair $(a)$, an alternating triple $(b)$ and a mixed alternating pair $(c)$.}}
\label{Altpairtriple}
\end{figure}

Given $v>u$ states of $D$, we write $D(v)_u$ for the chord diagram having the same circles as $D(v)$, but only the $1$-chords $e_i(v)$ such that $v_i \neq u_i$. 
\begin{remark}\label{miniremark}
If $w>v>u$ are states of $D$, then $e_i(v) \in D(v)_u$ implies that $e_i(w) \in D(w)_u$.
\end{remark}

\begin{proposition}\label{prop:char} Let $u$ be a state of $D$. Then, $\Phi(u)>0$ if and only if $D(\uno)_u$ contains a bichord or an alternating pair, and $\Phi(u)>1$ if and only if  $D(\uno)_u$ contains at least one of the following configurations:
\begin{enumerate}
\item Two non-parallel bichords.
\item An alternating pair and a bichord.
\item An alternating triple.
\item Two disjoint alternating pairs.
\item\label{it:612} A mixed alternating pair.
\end{enumerate}
\end{proposition}

The proof of the necessary condition (i.e., the implication $\Leftarrow$) is a case-by-case straightforward checking. We use Lemmas \ref{lemma:monochordbis} and \ref{lemma:bichordbis} in the proof of the sufficient condition.

\begin{lemma}\label{lemma:monochordbis} Let $w \succ v > u$ be three states with $\Phi(w)=\Phi(v)$. Then:
\begin{enumerate}
\item If $D(v)_u$ contains a bichord, then $D(w)_u$ contains a bichord or an alternating pair.
\item If $D(v)_u$ contains an alternating pair, then $D(w)_u$ contains an alternating pair.
\item If $D(v)_u$ contains an alternating triple, then $D(w)_u$ contains an alternating triple or a mixed alternating pair.
\item If $D(v)_u$ contains a mixed alternating pair, then $D(w)_u$ contains a mixed alternating pair.
\item If $D(v)_u$ contains an alternating pair and a bichord, then $D(w)_u$ contains an alternating pair and a bichord, or a mixed alternating pair or two disjoint alternating pairs.
\end{enumerate}
\end{lemma}

\begin{proof}
Since $\Phi(w)=\Phi(v)$, $w \succ_i v$ is a splitting for some monochord $e_i(w)$. The lemma follows from Remark \ref{miniremark} together with the following facts ($i \neq j \neq k \neq l$):
\begin{itemize}
\item[- ]\label{it:monochord:bichord} If $e_j(v)$ is a bichord, then either $e_j(w)$ is a bichord or $(e_j(w),e_i(w))$ form an alternating pair.
\item[- ]\label{it:monochord:alternatingpair} If $(e_j(v), e_k(v))$ is an alternating pair, then $(e_j(w), e_k(w))$ is an alternating pair.
\item[- ]\label{it:monochord:alttriple} If $(e_j(v),e_k(v),e_l(v))$ is an alternating triple, then either $(e_j(w),e_k(w),e_l(w))$ is an alternating triple, or $(e_i(w),e_j(w),e_k(w),e_l(w))$ is a mixed alternating pair.
\item[- ] If $(e_j(v),e_k(v),e_l(v),e_m(v))$ is a mixed alternating pair, then $(e_j(w),e_k(w),$ $e_l(w),e_m(w))$ is a mixed alternating pair. 
\item[- ]\label{it:monochord:altpairbichord} If $(e_j(v),e_k(v))$ is an alternating pair and $e_l(v)$ is a bichord, then either $(e_j(w),e_k(w))$ is an alternating pair and $e_l(w)$ is a bichord or $(e_i(w),e_j(w),$ $e_k(w),e_l(w))$ is a mixed alternating pair, or $(e_i(w),e_l(w))$ and $(e_j(w),e_k(w))$ are disjoint alternating pairs. \qedhere \end{itemize}\end{proof}

\begin{lemma}
\label{lemma:bichordbis} Let $w \succ v > u$ be three states with $\Phi(w)=\Phi(v)-1$. Then:
\begin{enumerate}
\item If $D(v)_u$ contains a bichord, then $D(w)_u$ contains two non-parallel bichords.
\item If $D(v)_u$ contains an alternating pair, then $D(w)_u$ contains an alternating triple or an alternating pair and a bichord.
\end{enumerate}
\end{lemma}

\begin{proof}
Since $\Phi(w)=\Phi(v)-1$, $w \succ_i v$ is a merging for some bichord $e_i(w)$. The lemma follows from Remark \ref{miniremark} together with the following facts ($i \neq j \neq k \neq l$):
\begin{itemize}
\item[- ]\label{it:bichord:bichord} If $e_j(v)$ is a \bichord, then $e_j(w)$ and $e_i(w)$ are non-parallel bichords.
\item[- ]\label{it:bichord:alternatingpair} If $(e_j(v),e_k(v))$ is an alternating pair, then either $(e_j(w),e_k(w))$ is an alternating pair and $e_i(w)$ is a bichord, or $(e_i(w),e_j(w),e_k(w))$ is an alternating triple. \qedhere \end{itemize} \end{proof}

\begin{proof}[Proof of the sufficient condition of Proposition \ref{prop:char}] Fix a chain $$\uno=u_0\succ_{i_1}u_1\succ_{i_2}\ldots\succ_{i_{k}} u_k= u.$$ 

Consider first the case $\Phi(u)>0$ and write $m$ for be the minimal number such that $u_{m-1}\succ_{i_m}u_m$ is a merging. Therefore, $e_{i_{m}}(u_{m-1}) \in D(u_{m-1})_u$ is a bichord and applying Lemma \ref{lemma:monochordbis} recursively, we deduce that $D(\uno)_u$ contains a bichord or an alternating pair. 

Consider now the case $\Phi(u)>1$ and write $m<p$ for the first two indices such that $u_{m-1}\succ_{i_m}u_m$ and $u_{p-1}\succ_{i_p}u_p$ are mergings. Therefore, $\Phi(u_m) = \Phi(u_{p-1})$, and applying recursively Lemma \ref{lemma:monochordbis} to the bichord $e_{i_{p}}(u_{p-1}) \in D(u_{p-1})_u$ we deduce that $D(u_m)_u$ contains a bichord or an alternating pair. Apply once Lemma \ref{lemma:bichordbis} and deduce that $D(u_{m-1})_u$ contains at least one of the following: two non-parallel bichords or an alternating triple or an alternating pair and a bichord. 

Finally, applying in each case recursively Lemma \ref{lemma:monochordbis} we deduce that $D(\uno)_u$ contains at least one of the following configurations:
\begin{enumerate}
\item Two non-parallel bichords.
\item An alternating pair and a bichord.
\item An alternating triple.
\item Two disjoint alternating pairs.
\item A mixed alternating pair.\qedhere
\end{enumerate} \end{proof}

\subsection{Ladybug set}

Let $u \in \cube{n}$ so that $\Phi(u)=1$. If $\flat{u} = \flat{\uno}-1$, recall from Lemma \ref{alternatingcycles}$(2a)$ that $G(u)$ contains a unique cycle. We define the \textit{ladybug set of $u$}, $\cL(u)$, as the singleton consisting of the connected component of $G(u)$ containing that cycle.

Consider now the case when $\flat{u} = \flat{\uno}$. Then, by Lemma~\ref{alternatingcycles}$(2b)$, $G(u)$ contains two cycles $\mathfrak{c}_1$ and $\mathfrak{c}_2$ sharing a common vertex $z$. Write $a_i,b_i$ for the edges of $\mathfrak{c}_i$ adjacent to $z$, for $i = 1,2$ (if $\mathfrak{c}_i$ is a loop, then $a_i=b_i$).

The chords $a_1,b_1,a_2,b_2$ divide the circle $z$ into four arcs (see Figure \ref{figladybug}). We define the \textit{ladybug set} $\cL(u)$ of $D(u)$ as the set whose elements are the two arcs of $z$ that are reached by traveling along any of the chords $a_1,b_1,a_2,b_2$ until $z$, and then taking a right turn.
\begin{figure}[h]
\centering
\includegraphics[width = 12.95cm]{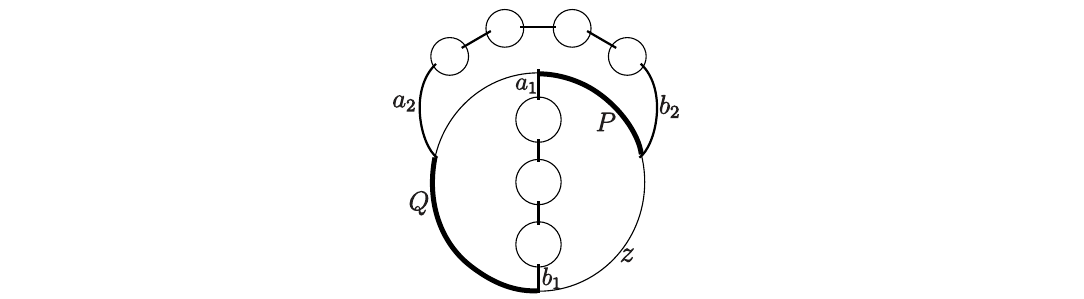}
\caption{\small{An example of $D(u)$ for the case when $\Phi(u)=1$ and $\flat{u} = \flat{\uno}$. The cycles $\mathfrak{c}_1$ and $\mathfrak{c}_2$ (of length 4 and 5, respectively) divide $z$ into four arcs. The ladybug set corresponds to the thickened arcs, i.e. $\cL(u) = \{P, Q\}$.}}
\label{figladybug}
\end{figure}

\begin{remark}
In the case when $\flat{u} = \flat{\uno}$ and $a_i=b_i$ for $i=1,2$, the set $\cL(u)$ coincides with the \textit{right pair} giving rise to the celebrated \textit{ladybug matching} defined in \cite{LSoriginal}.
\end{remark}

\begin{lemma}\label{lemma:ell} If $u>v$, $\flat{v}=\flat{\uno}$ and $\Phi(u)=\Phi(v)=1$, then each arc in $\cL(u)$ intersects one of the two arcs in $\cL(v)$ and viceversa.
\end{lemma}

\begin{proof}
Let $\Lambda =\{e_i(u) \, | \, u_i \neq v_i\} \subset \bar{E}(u)$. Since $\Phi(u) = \Phi(v)$, all elements in $\Lambda$ are loops and none pair among them constitute an alternating pair. 

By Lemma~\ref{alternatingcycles}$(2b)$, $G(u)$ contains two cycles $\mathfrak{c}_1$ and $\mathfrak{c}_2$ sharing a common vertex $z_u$. Write $z_v$ for the circle in $D(v)$ with the same property. Then $z_v$ is obtained from $z_u$ by performing surgery along the loops in $\Lambda$ with their endpoints in $z_u$. 

Write $\cL(u) = \{P,Q\}$ and let $\Lambda_P$ (resp. $\Lambda_Q$) be the subset of loops of $\Lambda$ so that at least one of their endpoints lies in $P$ (resp. $Q$). The disposition of $\mathfrak{c}_1$ and $\mathfrak{c}_2$ (stated in Lemma~\ref{alternatingcycles}$(2b)$) implies that $\Lambda_P \cap \Lambda_Q = \emptyset$. Define the following subsets of $\mathbb{R}^2$:
$$
p =\bigcup_{e_i(u)\in\Lambda_P} e_i(u) \quad \mbox{ and } \quad q =\bigcup_{e_i(u)\in\Lambda_Q} e_i(u).
$$

The endpoints of each loop $e_i(u) \in \Lambda_P$ (resp. $e_i(u) \in \Lambda_Q$) separate $z_u$ into two arcs: write $d_i$ for the unique arc in $z_u$ which is disjoint from $Q$ (resp. from $P$). Define the following subsets of $\mathbb{R}^2$: $$ d_P =\bigcup_{e_i(u)\in\Lambda_P} d_i \quad \mbox{ and } \quad d_Q =\bigcup_{e_i(u)\in\Lambda_Q} d_i.$$

As there are no alternating pairs in $\Lambda$, we deduce that
\begin{align*}
d_P\cap d_Q &=\emptyset & &\mbox{ and }& (p \cup d_P)\cap(q \cup d_Q)&=\emptyset  
\end{align*}
As a consequence, we can consider two disjoint open subsets $A,B$ of $\mathbb{R}^2$ such that
\begin{align*}
P \cup p\cup d_P&\subset A & Q\cup q\cup d_Q\subset B.
\end{align*}
Each of the arcs of $\cL(v)$ is obtained by doing surgery on each of the arcs of $\cL(u)$. The surgery performed on $P$ is supported in $A$, while the surgery perfomed on $Q$ is supported in $B$. Since $A \cap B = \emptyset$, it is possible to define $P'$ (resp. $Q'$) as the arc of $\cL(v)$ obtained from $P$ (resp. $Q$). Thus, $P$ intersects $P'$, $Q$ intersects $Q'$ and $P\cap Q' = Q \cap P' = \emptyset$, as desired.
\end{proof}

\subsection{The bijection $\rho_{u,v}$}

Given $u \succ_i v$ so that $\Phi(u) = \Phi(v) = 1$, then $\flat{u}=\flat{v}$, and therefore it is possible to define a bijection $\rho_{u,v}: \cL(u) \to \cL(v)$ in the following way:

\begin{enumerate}
\item[(1)] If $\flat{v} = \flat{\uno} -1$, then $\rho_{u,v}$ maps the unique element in $\cL(u)$ to the unique element in $\cL(v)$.
\item[(2)] If $\flat{v} = \flat{\uno}$, then by Lemma \ref{lemma:ell} each of the two arcs in $\cL(u)$ intersects one of the two arcs in $\cL(v)$. Write $\cL(u) = \{P_u, Q_u\}$ and $\cL(v)= \{P_v, Q_v\}$, labelling the arcs so that $P_u \cap P_v \neq \emptyset$ and $Q_u \cap Q_v \neq \emptyset$ as subsets of $\mathbb{R}^2$ (see Figure \ref{ladycases}). The bijection $\rho_{u,v}: \cL(u) \to \cL(v)$ is given by sending $P_u$ to $P_v$ and $Q_u$ to $Q_v$. In particular, in those cases when none of the endpoints of the chord $e_i(u)$ lie in $P_u$ nor $Q_u$, $\rho_{u,v}$ becomes the identity.
\end{enumerate}

\begin{figure}[t]
\centering
\includegraphics[width = 12.95cm]{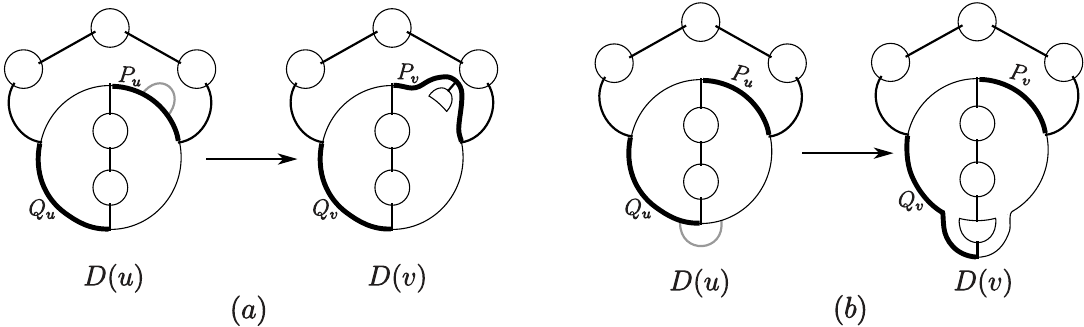}
\caption{\small{Two examples illustrating the bijection $\rho_{u,v}$ for the case $u \succ_i v$, $\Phi(u) = \Phi(v) = 1$ and $\flat{v} = \flat{\uno}$, with the edge $e_i(u)$ depicted in light colour. In both cases $\rho_{u,v}: \cL(u) \to \cL(v)$ maps $P_u$ to $P_v$ and $Q_u$ to $Q_v$. }}
\label{ladycases}
\end{figure}

The following lemma follows immediately from Lemma \ref{lemma:ell}:

\begin{lemma}\label{lematech1} Let $u \succ v, v' \succ w$ so that $\Phi(u)=\Phi(w)=1$. Then the following diagram commutes
\[\xymatrix{
\cL(u)\ar[r]^{\rho_{u,v}}\ar[d]^{\rho_{u,v'}} & \cL(v) \ar[d]^{\rho_{v,w}} \\
\cL(v')\ar[r]^{\rho_{v',w}} &\cL(w).
}\]
\end{lemma}

%
%

\subsection{The bijection $\mu_i(u)$}\label{Sec_hbarra}
Given $u\succ_{i}v$ with $\Phi(u)=0$, $\Phi(v)=1$ (i.e., a merging), we define a function 
\[\mu_i(u)\colon \mathcal{O}_i(u) \to \cL(v)\]
as follows:

\begin{enumerate}
\item[(1)]  If $\flat{v} = \flat{\uno}-1$, then $|\cL(v)|=1$ and we define $\mu_i(u)$ as the unique constant function.
\item[(2)] If $\flat{v} = \flat{\uno}$, write $\mathcal{O}_i(u)=\{z_1,z_2\}$, $\cL(v) = \{P,Q\}$, and assume without loss of generality that $P \cap z_1 \neq \emptyset$ and $Q \cap z_2\neq \emptyset$ as subsets of $\mathbb{R}^2$. Then, we define the bijection $\mu_i(u):\mathcal{O}_i(u) \to \cL(v)$ by declaring $\mu_i(u) (z_1) = P$, $\mu_i(u) (z_2) = Q$ (see Figure \ref{ladycases2}).
\end{enumerate}

\begin{figure}[t]
\centering
\includegraphics[width = 11.9cm]{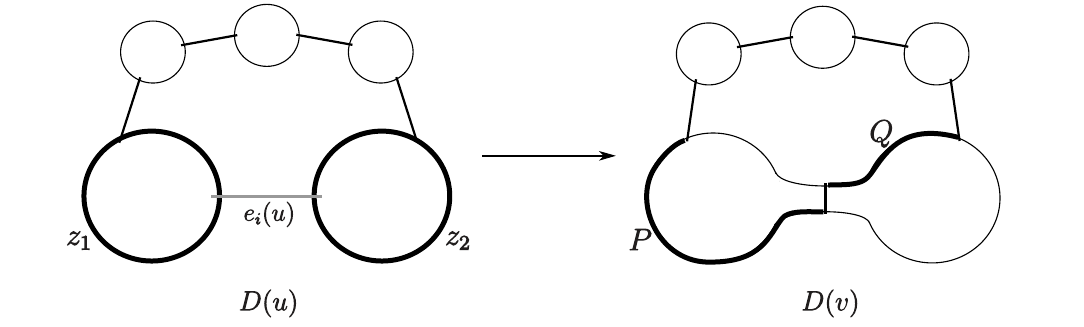}
\caption{\small{An example illustrating the bijection $\mu_i(u)$ for the case $u \succ_i v$, $\Phi(u) = 0$, $\Phi(v) = 1$ and $\flat{v} = \flat{\uno}$. The edge $e_i(u)$ is depicted in light colour. The bijection $\mu_i(u) \colon \mathcal{O}_i(u) \to \cL(v)$ sends $z_1$ and $z_2$ to $P$ and $Q$, respectively.}}
\label{ladycases2}
\end{figure}

The next result is a consequence of Lemma \ref{lemma:ell}.

\begin{lemma}\label{lemacuadrados} Let $u\succ_i v \succ_j w$, $u \succ_j v' \succ_i w$. Then, the following squares commute:
\begin{enumerate}
\item If $\Phi(u)=\Phi(v')=0$ and $\Phi(v)= \Phi(w)=1$, then\footnote{We defer the definition of $h_{u,v}\colon Z(u)\lra Z(v)$ some lines until the beginning of Section \ref{SecFunc}.}
\[\xymatrix{
\cO_i(u)\ar[r]^{\mu_i(u)}\ar[d]^{h_{u,v'}} & \cL(v) \ar[d]^{\rho_{v,w}} \\
\cO_i(v')\ar[r]^{\mu_i(v')} & \cL(w).
}\]
\item If $\Phi(u)=0$ and $\Phi(v)=\Phi(v')=\Phi(w)=1$, then
$$
\xymatrix{
\cO_j(u)=\cO_i(u)\ar[r]^-{\mu_i(u)}\ar[d]^{\mu_j(u)} & \cL(v)\ar[d]^{\rho_{v,w}} \\
\cL(v')\ar[r]^{\rho_{v',w}} & \cL(w).
}
$$
\end{enumerate}
\end{lemma}

\section{Khovanov functors}\label{SecFunc}

In this section we review the functor given by Lipshitz and Sarkar in \cite{LLS2017cubo,LLS15} and introduce a new Khovanov functor, giving a natural transformation between them which allows us to prove that the geometric realizations of both functors are homotopy equivalent at the almost-extreme quantum grading (Corollary \ref{MequivF}). First, we introduce some notation.

Let $u$ be a state of a diagram $D$ so that $\Phi(u)=0$. Orient $G(u)$ by fixing, for each of its edges, one of its two possible orientations. We define, for each $e\in E(u)$, the subgraph $e_+ \subset G(u)$ as the connected component of $G(u)\smallsetminus \{e\}$ towards which $e$ is pointing. Note that $e_+$ is well defined, since Lemma \ref{alternatingcycles} guarantees that $G(u)$ has no cycles.


Given any two states $u> v$ of $D$, there is an inclusion of the associated graphs $G(u) \subset G(v)$, and we define the maps
\begin{equation}\label{eq:2}f_{u,v}\colon \pi_0 G(u) \twoheadrightarrow \pi_0 G(v), \quad g_{u,v}\colon E(u)\hookrightarrow E(v); \end{equation}
observe that the first one is an isomorphism if $\Phi(u)=\Phi(v)$, while the second one is always injective. Additionally, if $u\succ_i v$, there exists a Burnside morphism
\begin{align*}
h_{u,v}\colon Z(u)\lra Z(v)
\end{align*}
mapping a circle $z\in Z(u)$ either to itself if $z \notin \cO_i(u)$, or to $\cO_i(v)$ if $z \in \cO_i(u)$.

\subsection{The Khovanov functor in almost-maximal grading}

The key piece in the construction of the Khovanov spectra in \cite{LLS15} was the \textit{Khovanov functor}, whose associated stable homotopy type is a link invariant. Using Lemma \ref{lematransformation}, we restate now this functor, adapted to the particular case of the almost-maximal quantum grading.

Given a link diagram $D$ with $n$ ordered crossings, consider the functor\footnote{When working with different diagrams, we write $F_D$ to denote the functor $F$ associated to the link diagram $D$.}
$$F \colon \cube{n}\lra \cB$$ 
defined, in a vertex $u \in \cube{n}$, as\footnote{The original functor $F$ from \cite{LLS15} splits into functors $F^j$ which associates to each state $u$ a set $\mathcal{E}_{u,j}$ of enhancements so that $q(u,x) = j$, for every $x\in \mathcal{E}_{u,j}$. When particularizing to the case when $j=j_{almax}$, Corollary \ref{corollaryL} implies that if $\Phi(u)=0$, then the set $\mathcal{E}_{u,j}$ can be identified with $Z(u)$, whereas if $\Phi(u)=1$ then $\mathcal{E}_{u,j}=\{u_+\}$ .  }:
\begin{align*}
F(u) &=
\begin{cases}
Z(u) & \text{if $\Phi(u)=0$}, \\
\{u_+\} & \text{if $\Phi(u)=1$}, \\
\emptyset & \text{if $\Phi(u) >1$}. \\
\end{cases}
\end{align*}

On a morphism $\varphi_{u,v}$, with $u\succ_i v$, the span $F_{u>v}:=F(\varphi_{u,v})\colon F(u)\to F(v)$ is given, for $z \in Z(u)$ and $u_+=(u, x_+)$, by
\begin{align*}
F_{u>v}(z) &=
\begin{cases}
h_{u,v}(z) & \text{if $\Phi(v)=0$}, \\
v_+ & \text{if $\Phi(v)=1$ and $z\in \cO_i(u)$},\\
\emptyset & \text{if $\Phi(v)=1$ and $z\notin \cO_i(u)$};
\end{cases}
\\
F_{u>v}(u_+) &=
\begin{cases}
v_+ & \text{if $\Phi(v)=1$}, \\
\emptyset & \text{if $\Phi(v)=2$}.
\end{cases}
\end{align*}

Finally, we specify $2$-morphisms: Let $u\succ_i v\succ_j w$ and $u\succ_j v'\succ_i w$. We need to produce a 2-morphism $F_{u,v,v',w}$ between the 1-morphisms $F_{v>w} \circ F_{u>v}$ and $F_{v'>w} \circ F_{u>v'}$.

Consider first the case when $\Phi(u)=0$. If $\Phi(u)=\Phi(v)=\Phi(v')=0$, $\Phi(w)=1$, then the chords $e_i(u)$ and $e_j(u)$ form an alternating pair attached to some circle $z_1$, so $\cO_j(v)=\cO_i(v)$ and $\cO_i(v')=\cO_j(v')$. If $z \in Z(u)$ and $z\neq z_1$, then $F_{v>w}\circ F_{u>v}(z) = \emptyset$. 

In the case $z=z_1$, label the circles involved in the mergings and splittings as
\begin{align*}
\{z_1\}=\cO_i(u) &= \cO_j(u), & \{z_2,z_3\} &= \cO_i(v)=\cO_j(v), \\
\{z_6\} = \cO_i(w)&=\cO_j(w), & \{z_4,z_5\} &= \cO_i(v')=\cO_j(v').
\end{align*}

Therefore,
\begin{align*}
F_{v>w}\circ F_{u>v}(z_1) &= F_{v>w}(z_2+z_3) = \{z_2\}\cdot w_+ + \{z_3\}\cdot w_+ = \cO_j(v)\cdot w_+,\\
F_{v'>w}\circ F_{u>v'}(z_1) &= F_{v'>w}(z_4+z_5) = \{z_4\}\cdot w_+ + \{z_5\}\cdot w_+ = \cO_i(v')\cdot w_+.
\end{align*}

Then, the 2-morphism $F_{u,v,v',w}$ consists of a bijection between $\cO_j(v)$ and $\cO_i(v')$ given by the ladybug matching $\mu_i(v')^{-1} \circ \mu_j(v): \cO_j(v) \leftrightarrow \cL(w) \leftrightarrow \cO_i(v')$ (see Section \ref{Sec_hbarra}). 

If $\Phi(u)=0$ and we are not in the previous situation, then every summand in the formal sums $(F_{v>w} \circ F_{u>v})(z)$ and $(F_{v'>w}\circ F_{u>v'})(z)$ has a singleton as coefficient, so there is a unique choice for the 2-morphism $F_{u,v,v',w}$.

In the case when $\Phi(u)=1$, then $F_{v,w} \circ F_{u,v}$ and $F_{v',w} \circ F_{u,v'}$ are empty if $\Phi(w)>1$; otherwise, both $F_{v,w} \circ F_{u,v}$ and $F_{v',w} \circ F_{u,v'}$ map $u_+$ to the formal sum $w_+$ which has a singleton as coefficient. 

\begin{lemma}\cite{LLS2017cubo}
$F$ satisfies conditions \ref{condition:matching} and \ref{condition:hexagon} in Lemma \ref{lematransformation}.
\end{lemma}

The (strictly unitary lax) 2-functor $F$ defined above is the Khovanov functor given in \cite{LLS15} at the almost-maximal quantum grading and, if $D$ has $n_-$ negative crossings, the almost-extreme (maximal) Khovanov spectrum is defined as 
\begin{equation}\label{eq:F}
\cX^{\jalmax}_D = \Sigma^{-n_-}\Sigma^\infty \Tot F_D.
\end{equation}

\subsection{A new equivalent functor}

We introduce now a (strictly unitary lax) 2-functor $M \colon \cube{n} \to \cB$, with the property that its realization coincides with the realization of $F$, as will be show in Corollary \ref{MequivF}. As before, we start by defining\footnote{When working with different diagrams, we write $M_D$ to denote the functor $M$ associated to the link diagram $D$.} $M$ for the vertices of the cube, corresponding to the states of $D$:

\begin{align*}
M(u) &=
\begin{cases}
\pi_0G(u)\cup E(u) & \text{if $\Phi(u)=0$,} \\
u_+ & \text{if $\Phi(u)=1$,} \\
\emptyset & \text{if $\Phi(u) >1$.}
\end{cases}
\end{align*}

On a morphism $\varphi_{u,v}$, with $u\succ_i v$, the span $M_{u>v}:=M(\varphi_{u,v})\colon M(u)\to M(v)$ is defined either on a connected component $C$ or an edge $e$ of $G(u)$ (if $\Phi(u) = 0$), or on the enhanced state $u_+$ (if $\Phi(u) = 1$), as follows:
\begin{align*}
M_{u>v}(C) &=
\begin{cases}
f_{u,v}(C) & \text{if $\Phi(v)=0$}, \\
(\cL(v)\cap f_{u,v}(C))\cdot v_+  & \text{if $\Phi(v)=1$}; \\
\end{cases}
\\
M_{u>v}(e) &=
\begin{cases}
g_{u,v}(e) & \text{if $\Phi(v)=0$}, \\
(\mu_i(u) (e_+ \cap \cO_i(u)) \cdot v_+  & \text{if $\Phi(v)=1$}; \\
\end{cases}
\\
M_{u>v}(u_+) &=
\begin{cases}
v_+ & \text{if  $\Phi(v)=1$}, \\
\emptyset & \text{if  $\Phi(v)=2$}.
\end{cases}
\end{align*}

\begin{remark}\label{rem1}
Observe that the expression $\cL(v)\cap f_{u,v}(C)$ equals either $\cL(v)$, in the case when $\cO_i(u) \cap C \neq \emptyset$, or the empty set otherwise.
\end{remark}

\begin{remark}\label{rem2}
Observe that the evaluation of the expression $e_+ \cap \cO_i(u)$ leads to either the empty set or one or two circles, as illustrated in Figure \ref{fige+}. When $e_+ \cap \cO_i(u) = z \in Z(u)$, there are two possibilities: if $\flat{v} = \flat{\uno}$, then $\mu_i(u) (e_+ \cap \cO_i(u))$ maps $z$ to one of the two arcs in $\cL(v)$; however, if $\flat{v} = \flat{\uno}-1$, then $\mu_i(u) (e_+ \cap \cO_i(u))$ maps $z$ to the singleton $\cL(v)$.  
\end{remark}

\begin{figure}[t]
\centering
\includegraphics[width = 10cm]{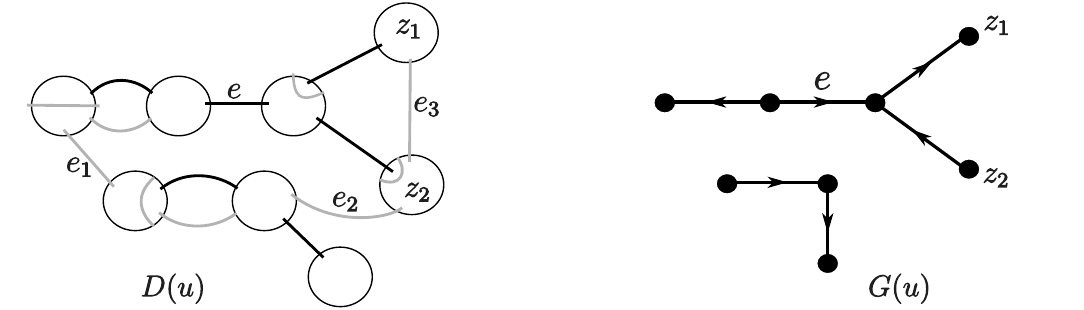}
\caption{\small{We illustrate the three cases considered in Remark \ref{rem2}. Given $D(u)$ and the orientation shown in $G(u)$, $e_+ \cap \cO_1(u) = \emptyset$, $e_+ \cap \cO_2(u) = \{z_2\}$ and $e_+ \cap \cO_3(u) = \{z_1,z_2\}$.}}
\label{fige+}
\end{figure}

Finally, given $u\succ_i v\succ_j w$ and $u\succ_j v'\succ_i w$, we define the 2-morphism $M_{u,v,v',w}$ between the 1-morphisms $M_{v>w} \circ M_{u>v}$ and $M_{v'>w} \circ M_{u>v'}$ as follows. 


First, note that $M_{u,v,v',w}$ is trivially defined when $\Phi(u) = \Phi(v)$ or $|\cL(w)|=1$, since it is a bijection between singletons. Then, we just need to specify the cases when $\Phi(u)=0$, $\Phi(w)=1$ and $|\cL(w)|=2$, depicted in Figure \ref{casesLu0Lw2}. Moreover, the value of $M_{v>w} \circ M_{u>v}$ and $M_{v'>w} \circ M_{u>v'}$ on a component $C \in \pi_0G(u)$ equals the empty set or a singleton, unless $f_{u,w}(C)$ is the component of $G(w)$ containing two cycles; a similar reasoning applies for edges $e \in G(u)$, unless $e$ (or possibly $g_{u,v}(e)$ or $g_{u,v'}(e)$) points to the two circles involved in the unique merging. We define $M_{u,v,v',w}$ in such components and edges:


\begin{figure}[h]
\centering
\includegraphics[width = 12.15cm]{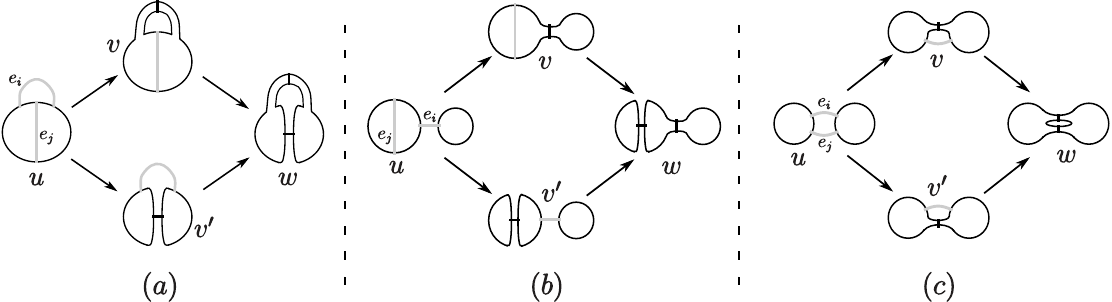}
\caption{\small{Cases ($a$), ($b$) and ($c$) in the definition of $M_{u,v,v',w}$. In order $|\cL(w)| = 2$, $\flat{w}$ must be equal to $\flat{\uno}$ and therefore the vertices corresponding to the two circles depicted in $D(u)$ in cases $(b)$ and $(c)$ must belong to the same connected component in $G(u)$.}}
\label{casesLu0Lw2}
\end{figure}

\begin{enumerate}
\item[($a$)] If $\Phi(u)=\Phi(v)=\Phi(v')=0$, $\Phi(w)=1$ and $|\cL(w)|=2$, then $e_i(u)$ and $e_j(u)$ form an alternating pair (see Figure \ref{casesLu0Lw2}(a)) and we have
\begin{align*}
M_{v>w}\circ M_{u>v}(C) = \cL(w) \cdot w_+ &\overset{\Id}{\longleftrightarrow} M_{v'>w}\circ M_{u>v'}(C) = \cL(w)\cdot w_+, \\
M_{v>w}\circ M_{u>v}(e) = \cL(w) \cdot w_+ &\overset{\Id}{\longleftrightarrow} M_{v'>w}\circ M_{u>v'}(e) = \cL(w) \cdot w_+.
\end{align*}

\item[($b$)] If $\Phi(u)=\Phi(v')=0$ and $\Phi(v)=\Phi(w)=1$ and $|\cL(w)|=2$, then $e_i(u)$ is a bichord and $e_j(u)$ a monochord (see Figure \ref{casesLu0Lw2}(b)) and we have
\begin{align*}
M_{v>w}\circ M_{u>v}(C) = \cL(v)\cdot w_+ &\overset{\rho_{v,w}}{\longleftrightarrow} M_{v'>w}\circ M_{u>v'}(C) = \cL(w)\cdot w_+, \\
M_{v>w}\circ M_{u>v}(e) = \cL(v)\cdot w_+ &\overset{\rho_{v,w}}{\longleftrightarrow} M_{v'>w}\circ M_{u>v'}(e) = \cL(w)\cdot w_+.
\end{align*}


\item[($c$)] If $\Phi(u)=0$ and $\Phi(v)=\Phi(v')=\Phi(w)=1$ and $|\cL(w)|=2$, then $e_i(w)$ and $e_j(w)$ are parallel bichords (see Figure \ref{casesLu0Lw2}(c)) and we have:
\begin{align*}
M_{v>w}\circ M_{u>v}(C) = \cL(v)\cdot w_+ &\overset{\rho_{v',w}^{-1}\rho_{v,w}}{\longleftrightarrow} M_{v'>w}\circ M_{u>v'}(C) = \cL(v')\cdot w_+, \\
M_{v>w}\circ M_{u>v}(e) = \cL(v)\cdot w_+ &\overset{\rho_{v',w}^{-1}\rho_{v,w}}{\longleftrightarrow} M_{v'>w}\circ M_{u>v'}(e) = \cL(v')\cdot w_+. 
\end{align*}
\end{enumerate}

\begin{lemma}\label{Mwelldefined}
The 2-functor $M$ satisfies conditions \ref{condition:matching} and \ref{condition:hexagon} in Lemma \ref{lematransformation} and therefore is well defined.
\end{lemma}

\begin{proof}
Condition \ref{condition:matching} follows from the definition of $M$. Consider an arbitrary three-dimensional face and denote its vertices as follows:
\begin{equation}\label{cubo01}
\xymatrix@C=1pc@R=1pc{
111\ar[dd]\ar[rr]^{M_{111>011}}\ar[dr] && 011\ar[dd]\ar[dr]^{M_{011>001}} & \\
&                                   101\ar[rr]\ar[dd] && 001 \ar[dd]^{M_{001>000}}\\
110\ar[rr]\ar[dr] && 010\ar[dr]&\\
&                    100\ar[rr] && 000. }\end{equation}

In order to prove condition \ref{condition:hexagon}, we need to show that $2$-morphisms in this cube commute. To do so, we study the bijections obtained when we \emph{move} along the cube. More precisely, starting from $M_{001>000}\circ M_{011>001}\circ M_{111>011}$ we move to $M_{010>000}\circ M_{011>010}\circ M_{111>011}$ and continue moving along the cube following Figure \eqref{movingcube2} until we reach again $M_{001>000}\circ M_{011>001}\circ M_{111>011}$. 

\begin{figure}[h]
\includegraphics[width=\textwidth]{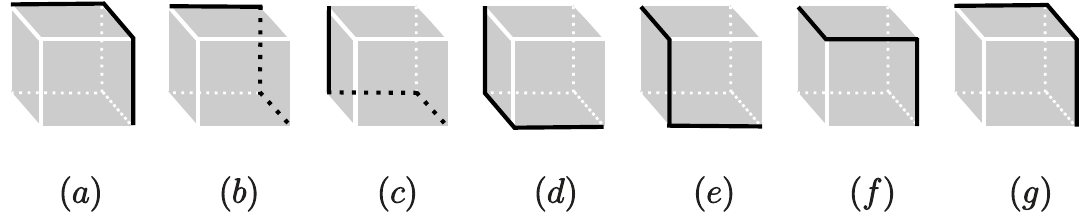}
\caption{Moving along the cube. Paths $(a)$, $(b)$ and $(c)$ represent $M_{001>000}\circ M_{011>001}\circ M_{111>011}$, \, $M_{010>000}\circ M_{011>010}\circ M_{111>011}$ and $M_{010>000}\circ M_{110>010}\circ M_{111>110}$, respectively.}\label{movingcube2}
\end{figure}

To simplify notation, write $\rho_i$ for the morphism $\rho_{u,v}$ if $u\succ_i v$; then, the commutative diagram of Lemma \ref{lematech1} associated to $u \succ_i v \succ_j w$ and $u \succ_j v' \succ_i w$ becomes:
\begin{equation}\label{eq:tech}\xymatrix{
\cL(u)\ar[r]^{\rho_i}\ar[d]^{\rho_j} & \cL(v)\ar[d]^{\rho_j} \\
\cL(v')\ar[r]^{\rho_i} & \cL(w).
}\end{equation}

We just need to study the cases when $\Phi(111) = 0$, $\Phi(000)=1$ and $|\cL(000)|=2$ (otherwise, the coefficients involved in all compositions of $1$-morphisms of the edges of the cube are singletons, and therefore the $2$-morphisms commute trivially). In principle, up to permutation, there are 8 different situations attending to the value of $\Phi$ in each vertex, depicted in Figure \ref{01}. However, it is not hard to check that situations $(a)$, $(e)$ and $(f)$ leads to inconsistency in the chord diagrams, and therefore they are not possible. We study the remaining 5 possible situations:  

\begin{figure}[h]
\includegraphics[width=\textwidth]{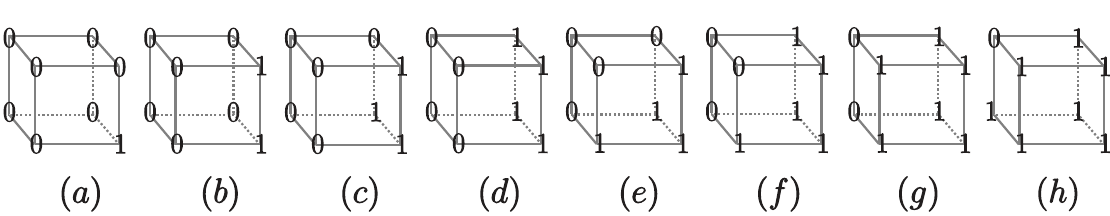}
\caption{Situations $(a)$, $(e)$ and $(f)$ are not possible when considering $\Phi(111) = 0$, $\Phi(000)=1$. Each $0$ or $1$ represent the value of $\Phi$ in each of the vertices of the cube \eqref{cubo01}.}\label{01}
\end{figure}

\begin{enumerate}
\item Figure \ref{01}$(b)$ (i.e., $\Phi(110) = \Phi(101) = \Phi(011) = \Phi(100) = \Phi(010) = 0 , \Phi(001) = 1$):  $D(111)_{000}$ contains three monochords, two of them constituting an alternating pair. Coefficients involved in the three-fold compositions when moving around the cube are either singletons or the following:
\begin{align*}
\cL(001)\overset{\rho_3}{\longleftrightarrow} \cL(000) = \cL(000) = \cL(000) = \cL(000) \overset{\rho_3^{-1}}{\longleftrightarrow} \cL(001) = \cL(001).
\end{align*}
The composition of the $2$-morphisms connecting coefficients is the identity morphism.

\item Figure \ref{01}$(c)$ (i.e., $\Phi(110) = \Phi(101) = \Phi(011) = \Phi(100) = 0 , \Phi(001) = \Phi(010)  = 1$): $D(111)_{000}$ contains three monochords constituting two alternating pairs. Coefficients involved in the three-fold compositions when moving around the cube are either singletons or the following: \begin{align*}
\cL(001)\overset{\rho_2^{-1}\rho_3}{\longleftrightarrow} \cL(010) = \cL(010) \overset{\rho_2}{\longleftrightarrow} \cL(000) = \cL(000) \overset{\rho_3^{-1}}{\longleftrightarrow} \cL(001) = \cL(001). \end{align*} 
The composition of the $2$-morphisms connecting coefficients is the identity morphism.

\item Figure \ref{01}$(d)$: $D(111)_{000}$ contains two monochords not forming an alternating pair and a bichord. Coefficients involved in the three-fold compositions when moving around the cube are either singletons or the following:
\begin{align*}
\cL(011) = \cL(011) \overset{\rho_3}{\longleftrightarrow} \cL(010) \overset{\rho_2}{\longleftrightarrow} \cL(000) = \cL(000) \overset{\rho_3^{-1}}{\longleftrightarrow} \cL(001) \overset{\rho_2^{-1}}{\longleftrightarrow} \cL(011). \end{align*}
The composition of the $2$-morphisms is the following square, which is a particular case of \eqref{eq:tech}:
\[
\xymatrix@C=1pc@R=1pc{
\cL(011)\ar[dd]^{\rho_3} &\\
&\cL(001)\ar[ul]_(.5){\rho_2^{-1}}\\
\cL(010)\ar[dr]^{\rho_2}& \\
 & \cL(000)\ar[uu]_(.5){\rho_3^{-1}}.
}\]

\item Figure \ref{01}$(g)$: $D(111)_{000}$ contains one monochord and two parallel bichords not forming an alternating triple. Coefficients involved in the three-fold compositions when moving around the cube are either singletons or the following:
\begin{align*}
\cL(011) = \cL(011) \overset{\rho_3}{\longleftrightarrow} \cL(010) \overset{\rho^{-1}_{1} \rho_2}{\longleftrightarrow} \cL(100) \overset{\rho_3^{-1}}{\longleftrightarrow} \cL(101) = \cL(101) \overset{\rho^{-1}_{2} \rho_1}{\longleftrightarrow} \cL(011). \end{align*}
The composition of $2$-morphisms is the boundary of the union of two squares of the form \eqref{eq:tech} along an edge, which commutes: 
\[
\xymatrix@C=1pc@R=1pc{
 && \cL(011)\ar[dd]_(.25){\rho_3} & \\
&                                   \cL(101)\ar[rr]^(.3){\rho_1} && \cL(001)\ar[lu]_(.35){\rho_2^{-1}} \\&& \cL(010)\ar[dr]^{\rho_2}&\\
&                    \cL(100) \ar[uu]^{\rho_3^{-1}} && \cL(000) \ar[ll]^{\rho_1^{-1}}. }\]

\item Figure \ref{01}$(h)$: $D(111)_{000}$ contains three parallel bichords. Coefficients involved in the three-fold compositions when moving around the cube are either singletons or the following:
\begin{align*}
\cL(011) = \cL(011) \overset{\rho_1^{-1}\rho_3}{\longleftrightarrow} \cL(110) = \cL(110) \overset{\rho^{-1}_3\rho_2}{\longleftrightarrow} \cL(101) = \cL(101) \overset{\rho_2^{-1}\rho_1}{\longleftrightarrow} \cL(011).
\end{align*}
The composition of $2$-morphisms is the boundary of the union of three squares of the form \eqref{eq:tech} along their common edges, which commutes:
\[\xymatrix@C=1pc@R=1pc{
                  && \cL(011)\ar[dd]_(.25){\rho_3} & \\
&                                   \cL(101)\ar[rr]^(.25){\rho_1} && \cL(001) \ar[ul]_{\rho_2^{-1}}\\
\cL(110)\ar[dr]_{\rho_2} && \cL(010)\ar[ll]_(.25){\rho_1^{-1}}&\\
&                    \cL(100)\ar[uu]_(.25){\rho_3^{-1}} &&
} \]\qedhere 
\end{enumerate}
\end{proof}

Once we have defined $M$, we will prove that its geometric realization is equivalent to that of the Khovanov functor $F$. To do so, we introduce a natural transformation $\gamma\colon M\to F$ and show that the induced homomorphism $\gamma_* \colon C_*(M; \mathbb{Z}) \to C_*(F; \mathbb{Z})$ is in fact an isomorphism (see Section \ref{ss:32}).

Given a state $u\in \cube{n}$, we set the natural natural transformation $\gamma$ as follows:
\begin{align*}
\gamma_u(C) &= \sum_{z\in C} z, \\
\gamma_u(e) &= \sum_{z\in e_+} z,\\
\gamma_u(u_+) &= u_+.\\
\end{align*}

Next, for each $u\succ_i v$, we define a $2$-morphism $\gamma_{u,v}: F_{u>v} \circ \gamma_u \to M_{u>v} \circ \gamma_v$, i.e., $\gamma_{u,v}$ makes the following diagram commute:
 \[
  \xymatrix@C=10ex{
    M(u)\ar[r]^-{\gamma_u}\ar[d]^-{M_{u>v}}&
    F(u)\ar[d]_-{F_{u>v}} \\
    M(v)\ar[r]^-{\gamma_v}&F(v).
  }
  \]
If $\Phi(u) = 0$ and $\Phi(v)=0$, $\gamma_{u,v}$ is defined as the identity:
\begin{align*}
F_{u>v}\circ \gamma_u(C) &= \sum_{z\in C} h_{u,v}(z) \overset{\Id}{\longleftrightarrow} \sum_{z\in f_{u,v}(C)}z \, =\, \gamma_v\circ M_{u>v}(C), \\
F_{u>v}\circ \gamma_u(e) &= \sum_{z\in e_+} h_{u,v}(z) \overset{\Id}{\longleftrightarrow} \sum_{z\in g_{u,v}(e)_+}z  \, =\, \gamma_v\circ M_{u>v}(e). 
\end{align*}
In the case when $\Phi(u) = 0$ and $\Phi(v)=1$ we define $\gamma_{u,v}$ as $\mu_i(u)$:
\begin{align*}
F_{u>v}\circ \gamma_u(C) &= (\cO_i(u)\cap C)\cdot v_+ \overset{\mu_i(u)}{\longleftrightarrow} (\cL(v)\cap f_{u,v}(C))\cdot v_+ = \gamma_v\circ M_{u>v}(C), \\
F_{u>v}\circ \gamma_u(e) &= (\cO_i(u)\cap e_+)\cdot v_+\overset{\mu_i(u)}{\longleftrightarrow} (\mu_i(u) (e_+ \cap \cO_i(u)) \cdot v_+ = \gamma_v\circ M_{u>v}(e).
\end{align*}
And when $\Phi(u) = 1 = \Phi(v)$, $\gamma_{u,v}$ is defined as the identity:
\begin{align*}
F_{u>v} \circ \gamma_u(u_+) &= v_+ \overset{\Id}{\longleftrightarrow} v_+  \, =\, \gamma_v \circ M_{u>v}(u_+). 
\end{align*}

\begin{lemma}\label{NatTransWellDefined}
The natural transformation $\gamma_u$ is well defined. 
\end{lemma}

\begin{proof}
We need to prove that $\gamma_u$ satisfies conditions \ref{condition:matching} and \ref{condition:hexagon} in Lemma~\ref{lematransformation}. Condition 
\ref{condition:matching} follows from definition of $\gamma_u$. Next, we prove \ref{condition:hexagon} by showing that $2$-morphisms in the following cube commutes, similar as we did in proof of Lemma~\ref{Mwelldefined}:

\[\xymatrix@C=1pc@R=1pc{
M(u) \ar[dd]^{\gamma_u}\ar[rr]\ar[dr] && M(v)\ar'[d][dd]_(-.25){\gamma_v}\ar[dr] & \\
&                                   M(v')\ar[rr]\ar[dd]_(.25){\gamma_v'} && M(w) \ar[dd]^{\gamma_w}\\
F(u)\ar'[r][rr]\ar[dr] && F(v)\ar[dr]&\\
&                    F(v')\ar[rr] && F(w). }\]

The commutativity is clear when all coefficients are singletons, i.e., it holds unless $\Phi(u)=0$, $\Phi(w)=1$ and $|\cL(w)|=2$; thus, we have to study the three following situations\footnote{The case $\Phi(u) = 0,  \Phi(v) = 0,  \Phi(v') = 1,  \Phi(w) = 1$ is symmetric to (\ref{Caso2Eles}).}, corresponding again to the cases illustrated in Figure \ref{casesLu0Lw2}:
\begin{align} \label{Caso1Eles}
\Phi(u) &= 0, & \Phi(v) &= 0, & \Phi(v') &= 0, & \Phi(w) &= 1; \\ \label{Caso2Eles}
\Phi(u) &= 0, & \Phi(v) &= 1, & \Phi(v') &= 0, & \Phi(w) &= 1;  \\ \label{Caso3Eles}
\Phi(u) &= 0, & \Phi(v) &= 1, & \Phi(v') &= 1, & \Phi(w) &= 1.  
\end{align}

Now, for each of the three cases above, we have to study the bijections obtained when we \textit{move} along the cube: starting from $\gamma_w \circ M_{v>w}\circ M_{u>v}$, we move to $F_{v>w}\circ \gamma_v\circ M_{u>v}$, and continue moving along the cube following Figure \ref{movingcube2}, until we reach again $\gamma_w \circ M_{v>w}\circ M_{u>v}$.

We just need to consider the action of those compositions on the component $C\in \pi_0G(u)$ such that $f_{u,w}(C)\in \pi_0G(w)$ contains two cycles (otherwise, the value of $F_{v>w}\circ \gamma_v\circ M_{u>v}$ is trivial, as explained in the proof of Lemma \ref{Mwelldefined}). 

Consider the case \eqref{Caso1Eles}, where edges $e_i(u)$ and $e_j(u)$ form an alternating pair (see Figure \ref{casesLu0Lw2}$(a)$). Then $F_{v>w}\circ \gamma_v\circ M_{u>v}(C)$ equals $\cL(w)\cdot w_+$, and we obtain the following values when moving along the cube as in Figure \ref{movingcube2}: 
$$ \cL(w)\cdot w_+ \overset{\mu_j(v)^{-1}}{\longleftrightarrow} \cO_j(v)\cdot w_+ = \cO_j(v)\cdot w_+ \overset{\mu_i(v')^{-1}\mu_j(v)}\longleftrightarrow \cO_i(v')\cdot w_+ = \cO_i(v')\cdot w_+$$ $$ \overset{\mu_i(v')}{\longleftrightarrow} \cL(w)\cdot w_+ = \cL(w)\cdot w_+.
$$

We focus now in case \eqref{Caso2Eles}, where $e_j(u)$ is a loop and $e_i(u)$ a bichord (Figure \ref{casesLu0Lw2}$(b)$ shows an example of this situation). In this case $F_{v>w}\circ \gamma_v\circ M_{u>v}(C)$ equals $\cL(v)\cdot w_+$, and we get:
$$ \cL(v)\cdot w_+ = \cL(v)\cdot w_+ \overset{\mu_i(u)^{-1}}{\longleftrightarrow} \cO_i(u)\cdot w_+ \overset{h_{u,v'}}{\longleftrightarrow} \cO_i(v')\cdot w_+ = \cO_i(v')\cdot w_+$$ $$\overset{\mu_i(v')}{\longleftrightarrow} \cL(w)\cdot w_+\overset{\rho_{v,w}^{-1}}{\longleftrightarrow} \cL(v)\cdot w_+.
$$

Consider the third case \eqref{Caso3Eles}, where $e_i(u)$ and $e_j(u)$ are parallel bichords, so $\cO_i(u)=\cO_j(u)$ (see Figure \ref{casesLu0Lw2}$(c)$). Then $F_{v>w}\circ \gamma_v\circ M_{u>v}(C) = \cL(v)\cdot w_+$, and we obtain:
$$ \cL(v)\cdot w_+ = \cL(v)\cdot w_+ \overset{\mu_i(u)^{-1}}{\longleftrightarrow} \cO_i(u)\cdot w_+ = \cO_j(u)\cdot w_+ \overset{\mu_j(u)}{\longleftrightarrow} \cL(v')\cdot w_+ = \cL(v')\cdot w_+$$ $$\overset{\rho_{v,w}^{-1}\rho_{v',w}}{\longleftrightarrow} \cL(v)\cdot w_+.
$$

Hence, showing the commutativity of the cube reduces to proving the following equalities in the corresponding situations:
\begin{align*}
\mu_i(v') \circ \mu_i(v')^{-1} \circ \mu_j(v) \circ \mu_j(v)^{-1} &= \Id, \\
\rho_{v,w}^{-1}  \circ \mu_i(v') \circ h_{u,v'}\circ \mu_i(u)^{-1} &= \Id, \\
\rho_{v,w}^{-1}\circ \rho_{v',w} \circ \mu_j(u)\circ \mu_i(u)^{-1} &=\Id.
\end{align*}

The first line holds on the nose, whereas the second and third are a consequence of the commutativity of the following squares, proved in Lemma \ref{lemacuadrados}:
\begin{eqnarray*}
\xymatrix{
\cO_i(u)\ar[r]^{\mu_i(u)}\ar[d]^{h_{u,v'}} & \cL(v) \ar[d]^{\rho_{v,w}} \\
\cO_i(v')\ar[r]^{\mu_i(v')} & \cL(w),
} &
\xymatrix{
\cO_i(u)\ar[r]^{\mu_i(u)}\ar[d]^{\mu_j(u)} & \cL(v) \ar[d]^{\rho_{v,w}} \\
\cL(v')\ar[r]^{\rho_{v',w}} & \cL(w).
}\end{eqnarray*}
\end{proof}


\begin{proposition}\label{propisom} $\gamma_*\colon C_*(M;\bZ)\to C_*(F;\bZ)$ is an isomorphism of chain complexes.
\end{proposition}
\begin{proof} 
Given a state $u\in \cube{n}$ with $\Phi(u)=0$ and a vertex $z\in Z(u)$ in $G(u)$, write $C_z$ for the component of $G(u)$ containing $z$ and write $E_z(u)$ for the set of edges incident to $z$. For each edge $e\in E_z(u)$, set
\begin{align*}
\sigma(e,z) & = \begin{cases} 0 & \mbox{if } z \in e_+, \\
1 & \mbox{if } z \notin e_+,
\end{cases} & 
\bar{\sigma}(e,z) & = \begin{cases} 1 & \mbox{if } z \in e_+, \\
0 & \mbox{if } z \notin e_+.
\end{cases} 
\end{align*}

The inverse of $\gamma_*$ is:
\begin{align*}
\gamma_*^{-1}(z) &= \sum_{e\in E_z(u)} (-1)^{\sigma(e,z)}e +\left(1-\sum_{e\in E_z(u)}\bar{\sigma}(e,z)\right) C_z,  \\
\gamma_*^{-1}(u_+) &=u_+.  \qedhere
\end{align*} 

\end{proof}

The above proposition together with Remark \ref{remark:spectraequiv} yield to the following result:

\begin{corollary}\label{MequivF} The spectra $\Tot F$ and $\Tot M$ are homotopy equivalent.
\end{corollary}

\section{Pointed semi-simplicial sets}

\begin{df}
A link diagram $D$ is \emph{$1$-adequate}  (resp. \emph{$0$-adequate}) if $\mathcal{G}(\uno)$ (resp. $\mathcal{G}(\cero)$) contains no loops. $D$ is said to be \emph{adequate} if it is both $0$-adequate and $1$-adequate. If $D$ is either $0$-adequate or $1$-adequate, it is called semiadequate. A link is said to be (semi)adequate if it admits a (semi)adequate diagram. 
\end{df}

\begin{proposition}\label{prop:3} The functor $F$ factors through $\Setp$ if and only if the diagram $D$ is $1$-adequate. The functor $M$ factors through $\Setp$ if and only if $D(\uno)$ contains no alternatig pairs.
\end{proposition}
\begin{proof} 
From the definition of $F$ it follows that $F_{u>v}$ is not a function of pointed sets if and only if $\Phi(u)=\Phi(v)=0$. This situation is avoided for all possible states $u > v$ if and only if $\Phi(v) > 0$ for every state $v \neq \uno$, or equivalently, if and only if $\mathcal{G}(\uno)$ contains no loops. 

Now, looking at the definition of $M$, we get that $M_{u>v}$ is not a function of pointed sets if and only if $\Phi(u)=0$, $\Phi(v)=1$ and $|\cL(v)|=2$. This situation is avoided for all possible states $u>v$ if and only if $D(\uno)$ contains no alternating pairs.
\end{proof}

Combining the functor $\Lambda$ from \eqref{Lambda} with the definition in \eqref{eq:F} we get the following results:

\begin{corollary}\label{corPS} If $D$ is 1-adequate, then $\cX_D^\jalmax$ is an iterated desuspension of the augmented semisimplicial pointed set $\Lambda(F)$ (cf. \cite{PS}).
\end{corollary}

\begin{corollary} If $D(\uno)$ has no alternating pairs, then $\cX_D^\jalmax$ is an iterated desuspension of the augmented semisimplicial pointed set $\Lambda(M)$.
\end{corollary}

\section{Breaking up $M_D$}

In this section we introduce skein sequences for $M_D$ (Section \ref{skein}). To do so, we decompose $M_D$ into subposets of the cube $\cube{n}$ (Section \ref{section:decomposition}), and compare one of them with the simplicial complex $I_D$ introduced in \cite{GMS} for extreme Khovanov homology. 

\subsection{The simplicial complex $I_D$} Given a link diagram $D$, in \cite{GMS} authors introduce a simplicial complex whose associated cohomology complex coincides with the Khovanov homology of the link in the minimal quantum grading. We restate that construction in terms of the maximal quantum grading $j_{\max}$. 

\begin{df}\label{DefLando}
Let $D$ be a link diagram. Its associated Lando graph $G_L(D)$ is constructed from $D(\uno)$ by considering a vertex for every monochord, and an edge joining two vertices if the endpoints of the corresponding monochords alternate along the same circle. We define the independence complex\footnote{The independence complex is defined in terms of a graph. Hence, given a graph $G$ we define its associated independence complex $I_G$ as described in Definition \ref{DefLando}.} associated to $D$, $I_D$, as the simplicial complex whose set of vertices is the same as the set of vertices of $G_L(D)$ and $\sigma=(e_1 e_2 \cdots e_k)$ is a simplex in $I_D$ if and only if the vertices $e_1, e_2, \ldots, e_k$ are independent in $G_L(D)$, i.e., there are not edges in $G_L(D)$ between these vertices. 
\end{df}

Note that $G_L(D)$ can be thought as the disjoint union of the Lando graphs arising from each of the chord diagrams in $D(\uno)$ after removing all bichords. Moreover, if $G_L(D) = G_L(D_1) \sqcup G_L(D_2)$, then $I_D = I_{D_1} * I_{D_2}$, the join of both simplicial complexes. 

\begin{theorem}\cite{GMS}
Let $L$ be an oriented link represented by a diagram $D$ with $p$ positive crossings. Then $$Kh^{i,j_{\max}}(D) \approx \tilde{H}_{p-i-1}(I_D).$$
\end{theorem}

The poset of faces of the independence complex $I_D$ is precisely the subposet of the cube $\cube{n}$ of those states $u$ for which $\Phi(u)=0$. 

Given a set of vertices $\{v_1,\ldots,v_n\}$, each simplicial complex on these vertices gives rise to a downwards closed subposet of $\cube{n}$ (i.e., if a state $u$ belongs to the subposet and $u > v$, then $v$ belongs to the subposet too): its poset of faces. Conversely, every downwards closed subposet of $\cube{n}$ is the poset of faces of some simplicial complex.

Let $\mathbf{S}\subset \Set$ be the full subcategory on $\emptyset$ and a singleton, and let $\mathbf{S_p}\subset \Setp$ be the full subcategory on the basepoint and $S^0$. Downwards closed subposets of the cube are in bijection with functors $\cube{n}\to \mathbf{S}$. Subposets of the cube are in bijection with functors $\cube{n}\to \mathbf{S_p}$ (see discussion in the final section of \cite{CS}, for example). The realization of a subposet of the cube is the desuspension of the totalization of its associated functor $\cube{n}\to \mathbf{S_p}\subset \Setp$. Thus, the realization of a simplicial complex $X$ coincides with the realization of its poset of faces.

Given $u\in \cube{n}$, we write $\bar{u}$ for the state of $\cube{n}$ satisfying $\bar{u}_i \neq u_i$, for $1\leq i \leq n$.

\begin{df} The \emph{categorical dual} of a downwards (upwards) closed subposet $X\subset \cube{n}$ is the upwards (downwards) closed subposet $X^*\subset \cube{n}\to \mathbf{S_p}$ given by $u\in X^*$ if and only if $\bar{u}\in X$.

The \emph{complement} of a downwards (upwards) closed subposet $X\subset \cube{n}$ is the upwards (downwards) closed subposet $\hat{X}\subset \cube{n}$ given by $u\in \hat{X}$ if and only if $u\notin X$.
\end{df}

The complement of the categorical dual of a downwards closed set $X$ is again a downwards closed subposet of the cube whose associated simplicial complex is the \emph{Alexander dual} of $X$. In general $|\hat{X}|\simeq \Sigma|X|$ if $X$ is downwards closed and $|X^*|$ is Spanier-Whitehead dual to $|X|$. Observe that
\begin{enumerate}
\item If $|X|\simeq |A|\vee |B|$, then $|X^*|\simeq |A^*|\vee|B^*|$,
\item If $|X|\simeq S^k$, then $|X^*|\simeq S^{n-k-2}$, where $n$ is the dimension of the cube. 
\end{enumerate}

Given a link diagram $D$, the upwards closed subposet of the cube given by those states $u$ such that $\Phi(u)=0$, which we denote $X_D$, corresponds to the functor $F^{j_{\max}}\colon \cube{n}\to \Setp \subset \cB$. 

\begin{corollary}\label{corID=XD}
$X_D$ is the categorical dual of the poset of faces of $I_D$. As a consequence if  \, $|I_D| \simeq S^k$, then $|X_D| \simeq S^{n-k-2}$.
\end{corollary}

\subsection{Decomposing $M_D$}\label{section:decomposition}
Given a state $u$ of a diagram $D$, recall that we write $D(\uno)_u$ for the chord diagram having the same circles as $D(\uno)$ and those chords $e_i(\uno)$ so that $u_i \neq 1$. Given a crossing $c$ in $D$, we write $D[c=0]$ (resp. $D[c=1]$) for the link diagram obtained after smoothing $c$ following a $0$-label (resp. $1$-label). 

\begin{df}
Given a link diagram $D$ with $n$ crossings, we consider the following subposets of the cube $\cube{n}$:
\begin{itemize}
\item $X_D$: is the subposet of $\cube{n}$ consisting of those states $u$ so that $\Phi(u)=0$ (as defined in previous section). 
\item $X_D^e= X_{D[e=0]}$, for any monochord $e \in D(\uno)$. 
\item $Y_D$: is the subposet of $\cube{n}$ consisting of those states $u$ so that $\Phi(u)=1$ and $D(\uno)_u$ contains an alternating pair. 
\item $Z_D^b$:  is the subposet of $\cube{n}$ consisting of those states $u$ so that $\Phi(u)=1$ and $D(\uno)_u$ contains a bichord parallel to a given bichord $b$ of $D(\uno)$.
\end{itemize}
\end{df}

Observe that each state (admitting an enhacement) in the almost-extreme complex of $D$ belongs to one and only one of the previous subcubes $X_D, Y_D$ or $Z_D$ (see Corollary \ref{corollaryL} and Proposition \ref{prop:char}). 

\begin{proposition}\label{propcofibreseq}
The functor $M_D$ can be realized as the following cofibre sequence
\begin{equation}\label{eq:pushout}
\displaystyle\Sigma^{-1}\left(\bigvee_{Z(\uno)}X_D\vee \bigvee_{e\in N} X_D^e\right) \lra
\displaystyle Y_D\vee \bigvee_{[b]\in B} Z^b_D\lra M_D,
\end{equation}
where $N$ is the set of monochords in $G(\uno)$ and $B$ the set of classes of parallel bichords.
\end{proposition}

The decomposition follows because the second map is levelwise injective and the suspension of the leftmost term is precisely the quotient functor of this second map.

\begin{remark}\label{remarkMayerVietoris}
A Mayer-Vietoris spectral sequence allows to compute both posets $Y_D$ and $Z_D^b$ in terms of some posets $X_{D'}$, where $D'$ are partially smoothed link diagrams of $D$. We explain this below. 
\end{remark}

The poset $Y_D$ can be covered with the subposets $X_{D[e=0,f=0]}$ where $\{e,f\}$ runs along the set of all alternating pairs in $D(\uno)$. The $k$-fold intersections $A_k$ of this covering are 
\[A_k = \bigvee_{\{e_1,\ldots,e_{k+1}\}} X_{D[e_1=0,\ldots,e_{k+1}=0]},\]
where $\{e_1,\ldots,e_{k+1}\}$ are monochords attached to the same circle in such a way that they can be divided into two non-empty sets $H_1$ and $H_2$ so that all monochords in $H_1$ alternate with all monochords in $H_2$ and viceversa. Similarly, the poset $Z_D^b$ can be covered with the subposets $X_{D[b=0]}$ where $b$ runs along the set of bichords of $D(\uno)$, and the $k$-fold intersections $B_k$ of this covering are
\[B_k = \bigvee_{\{b_1,\ldots,b_k\}}X_{D[b_1=0,\ldots,b_k=0]},\] 
where $\{b_1,\ldots,b_k\}$ are bichords parallel to $b$. 

Thus, the (first page of the) Mayer-Vietoris spectral sequences associated to these coverings are: 
\begin{align*}
H^p\left(\bigvee_{\{e_1,\ldots,e_{q+1}\}}X_{D[e_1=0,\ldots,e_{q+1}=0]}\right) &\Rightarrow H^{p+q}(Y_D),\\
H^p\left(\bigvee_{\{b_1,\ldots,b_q\}}X_{D[b_1=0,\ldots,b_q=0]}\right) &\Rightarrow H^{p+q}(Z^b_D).
\end{align*}

\begin{example}
Consider the torus knot $T(3,q)$ and let $D=D_{(3,q)}$ be its standard diagram. Since $D(\uno)$ contains no bichords, $Z_D^b$ is trivial. We will combine Remark \ref{remarkMayerVietoris} together with Corollary \ref{corID=XD} to compute the realization of  subposets in \eqref{eq:pushout}. More precisely, we will express $X_D$, $X_D^e$ and $Y_D$ (via $A_k$) as the duals of independence complexes of some Lando graphs. Write $C_{n}$ for the cycle graph of $n$ vertices and $L_n$ the path of length $n$.

\begin{itemize}
\item[-] $G_L(D) = C_{2q}$;
\item[-] $G_L(D[e=0]) = L_{2q-4}$, for any monochord $e$ in $D(\uno)$;
\item[-] $G_L(D[e_1=0,e_2=0]) = C_{2q-2}$ when $e_1$ and $e_2$ are two consecutive monochords in $D(\uno)$. Note that alternating pairs consist precisely in consecutive monochords;
\item[-] $G_L(D[e_1=0,e_2=0,e_3=0]) = L_{2q-6}$ when $e_1,e_2,e_3$ are three consecutive monochords in $D(\uno)$. Note that the previous condition is equivalent to require that $e_1,e_2,e_3$ can be divided into two subsets $H_1$ and $H_2$ such that all monochords in each subset alternate with all monochords in the other subset. 
\end{itemize}
Note that there are no $n$-tuples of monochords $e_1,\ldots,e_n$ which can be divided into two subsets $H_1$ and $H_2$ satisfying the condition above when $n>3$. 

Summarizing we have:
\begin{align*}
X_D &=I^*_{C_{2q}} &
X_D^e &= I^*_{L_{2q-4}}&
A_k &= \begin{cases}
\bigvee_{2q} I^*_{C_{2q-2}} & \text{if $k=1$},\\
\bigvee_{2q} I^*_{L_{2q-6}} & \text{if $k=2$}, \\
\emptyset & \text{if $k>2$.}
\end{cases}
\end{align*}

The homotopy types of the above complexes were computed in \cite[Corollary 3.4 and Proposition 3.9]{JACO}:
\begin{align*}
|I_{C_{n}}| &\simeq \begin{cases}
S^{k-1} &\text{if $n=k\pm 1$},\\
S^{k-1}\vee S^{k-1} &\text{if $n=3k$}.
\end{cases}
&
|I_{L_{n}}| &\simeq \begin{cases}
S^{k} &\text{if $n=3k+1,3k+2$}, \\
* &\text{if $n=3k$}.
\end{cases}
\end{align*}

This allows to compute the almost-extreme Khovanov spectra of $T(3,q)$ in terms of well-known independence complexes of cycles and paths, with no need to apply induction on $q$.
\end{example}

\begin{remark}\label{remark:cone} The cone-length $c(\cX)$ of a spectrum $\cX$ is the least $n$ such that there is a sequence of cofibre sequences 
$$\cY_i\to \cX_i\to \cX_{i+1}$$
for $i=0,\ldots, n-1$, such that $\cX_0$ is contractible, $\cY_i$ is a wedge of spheres and $\cX_n = \cX$. 
In \cite{JACO} it was conjectured that $|X_D|$ is homotopy equivalent to a wedge of spheres for any diagram $D$. If this were true, then Remark \ref{remarkMayerVietoris} would imply that the cone length of $\Tot M_D$ is bounded above by the maximum of the following numbers:
\begin{enumerate}
\item the maximum number of parallel bichords plus one in $D(\uno)$.
\item the maximum number of monochords in $D(\uno)$ that can be partitioned into two disjoint subsets $H_1,H_2$ such that all chords in $H_1$ alternate with all chords in $H_2$ and viceversa.
\end{enumerate}
\end{remark}

\subsection{Skein sequences}\label{skein}

Let $a$ be a crossing in a link diagram $D$. Observe that if $a$ is a monochord in $D(\uno)$ we get the skein short exact sequence\footnote{Compare to skein short exact sequences from \cite{ViroFundamenta}.}
\begin{equation}\label{skeinloop}M_{D[a=1]}\lra M_{D[a=0]}\lra M_D, \end{equation} 
whereas if $a$ is a bichord in $D(\uno)$, then the skein short exact sequence becomes
\begin{equation}\label{skeinbichord} M_{D[a=1]}\lra X_{D[a=0]}\lra M_D. \end{equation}

Moreover, when $a$ is a monochord in $D(\uno)$ we also have the following sequence:
\begin{equation}\label{skeinX}X_{D[a=1]}\lra X_{D[a=0]}\lra X_D. \end{equation}

To verify the three above sequences note that the second map is an inclusion whose quotient is the suspension of the first spectrum. 


\begin{df} Let $D$ be a diagram and let $a$ be a monochord in $D(\uno)$. We say that $a$ is:
\begin{enumerate}
\item \emph{$2$-free} if it is not part of any alternating pair in $D(\uno)$. 
\item \emph{$3$-free} if it is not part of any alternating triple in $D(\uno)$. 
\item \emph{free} if it is $2$-free and $3$-free.
\item \emph{$b$-free} if it is not part of any alternating triple involving the bichord $b$ in $D(\uno)$. 
\end{enumerate}
\end{df}

\begin{lemma}\label{lemma:2free} Let $D$ be a diagram and let $a$ be a monochord in $D(\uno)$. 
\begin{enumerate}
\item\label{2free:1} If $a$ is $2$-free, then $X_D,Y_D$ and all $X_D^e$, with $e\neq a$ a monochord, are contractible; if, additionally, $D(\uno)$ contains another $2$-free monochord, then $X^a_D$ is contractible too.
\item\label{2free:3} If $a$ does not form an alternating triple with a bichord $b$ (i.e., $a$ is $b$-free), then $Z^b_D$ is contractible.
\end{enumerate}
\end{lemma}
\begin{proof} 
The leftmost map in skein sequence \eqref{skeinX} along the crossing associated to $a$ is identity, hence $X_D$ is contractible. A similar reasoning works for $X_D^e$ when $e \neq a$.

In general, the skein sequence \eqref{skeinloop} does not respect the decomposition in \eqref{eq:pushout}; however, when $a$ is $2$-free, it respects the term $Y_D$ and when $a$ is $b$-free it respects $Z_D^b$. In both cases the maps $Y_{D[a=1]} \to Y_{D[a=0]}$ and $Z^b_{D[a=1]} \to Z^b_{D[a=0]}$ are the identity. 
\end{proof}

\begin{corollary}\label{soloZ}
If $D(\uno)$ contains at least two 2-free monochords, then $$M_D \simeq \bigvee_{[b]\in B} Z^b_D.$$
\end{corollary}
\begin{proof}
The result is direct from cofibre sequence \eqref{eq:pushout} and Lemma \ref{lemma:2free}\eqref{2free:1}.
\end{proof}

\begin{corollary}\label{corollary:freemonochord} If $D$ has a free monochord $e$, then $M_D\simeq X_{D[e=0]}$.
\end{corollary}

\begin{proof} Lemma \ref{lemma:2free} implies that the only possibly non-contractible subposet in \eqref{eq:pushout} is $X_D^e = X_{D[e=0]}$. 
%
\end{proof}

\begin{lemma}\label{lemma:noregions1} 
Let $a$ be a 2-free monochord in $D(\uno)$ dividing a circle into two regions so that both of them contain at least one 2-free monochord (we say that $D(\uno)$ contains nested monochords). Then $M_D \simeq \Sigma M_{D[a=1]}$.
\end{lemma}
\begin{proof}
We use skein sequence \eqref{skeinloop} and show that $M_{D[a=0]}$ is contractible. 
First, notice that the circle in the statement is splitted into two circles $z_1$ and $z_2$ when changing the $1$-label of $a$ for a $0$-label, each of them attached to at least a $2$-free monochord that we call $c$ and $d$, respectively. Lemma \ref{lemma:2free}\eqref{2free:1} implies that the the only possibly non-contractible spaces when applying cofibre sequence \eqref{eq:pushout} to $D_{[a=0]}$ are $Z_{D[a=0]}^{b}$.

Now, since $a$ is 2-free in $D(\uno)$, there is no bichord connecting $z_1$ and $z_2$ in $D[a=0](\uno)$. Therefore, for any bichord $b$ in $D[a=0](\uno)$ at least one of the monochords $c$ or $d$ are $b$-free, so Lemma \ref{lemma:2free}\eqref{2free:3} implies that $Z_{D[a=0]}^b$ is contractible for all bichords $b$, thus $M_{D[a=0]}$ is contractible. 
\end{proof}

\begin{df}
Two parallel bichords $a$ and $b$ of a chord diagram $D(\uno)$ are equivalent if there is no monochord $e$ so that $a,b$ and $e$ constitute an alternating triple. 
\end{df}

\begin{lemma}\label{lemma:equivalent} If $D$ has two equivalent bichords $a,b$, then $M_D \simeq \Sigma M_{D[a=1]}$.   
\end{lemma}

\begin{proof} Since $a$ and $b$ are equivalent, the chord $b$ is $2$-free in the chord diagram $D[a=0](\uno)$, so Lemma \ref{lemma:2free}(1) implies that $X_{D[a=0]}$ is contractible. The statement follows from the skein sequence \eqref{skeinbichord} along the chord $a$.
\end{proof}

\section{Diagrams with no alternating pairs}

In this section we determine the homotopy type of $M_D$ (thus, that of the Khovanov spectrum given by Lipshitz and Sarkar in \cite{LSoriginal} for the almost-extreme quantum grading) for diagrams $D$ so that $D(\uno)$ contains no alternating pairs. Observe that this determines their Khovanov homology groups at almost-extreme quantum degree.

\subsection{Diagrams with no monochords} 
This case was studied in \cite{PS}: If $D$ is $1$-adequate (i.e., $D(\uno)$ has no monochords), then $M_D$ is  homotopy equivalent to a wedge of spheres and possibly a suspension of the projective plane. See \cite[Theorem 1.1]{PS} for the details.  

\subsection{Diagrams with one monochord} 

\begin{proposition}
Let $D$ be a link diagram with $n$ crossings so that $D(\uno)$ contains a single monochord attached to a circle $z$, and let $k$ be the number of circles connected to $z$ along an alternating triple. Then, $$|M_D| \simeq \bigvee_{k-1}S^{n-3}.$$
\end{proposition}

\begin{proof}
Write $e$ for the monochord in $D(\uno)$. By Lemma \ref{lemma:2free}\eqref{2free:1}, the decomposition in \eqref{eq:pushout} becomes
\[\Sigma^{-1} X^e_D\lra \bigvee_{[b]\in B} Z^b_D\lra M_D.\]
The subposet $X^e_D$ has a single element in degree $n-1$, thus $|X^e_D|\simeq S^{n-2}$. By Lemma \ref{lemma:2free}\eqref{2free:3}, $Z^b_D$ is contractible unless $b$ forms an alternating triple with $e$. Therefore we may assume, up to suspension, that all bichords are attached to $z$ and form an alternating triple with $e$. Moreover, using Lemma \ref{lemma:equivalent} we can assume, up to suspension, that circles are connected to $z$ by exactly two bichords.  

Therefore, for each such bichord $b$, the subposet $Z^b_D$ has five elements in M-shape, with two elements in degree $n-1$ and three in degree $n-2$, thus $|Z^b_D| \simeq S^{n-3}$, and the above decomposition becomes
\[S^{n-3}\lra \bigvee_{[b]\in B}S^{n-3}\lra |M_D|.\]
The map to the wedge of spheres is a diagonal map. This concludes the proof.
\end{proof}

\subsection{Diagrams with 2-free monochords}\label{Sec2-free}
Let $D$ be a diagram whose associated chord diagram $D(\uno)$ contains more than one monochord, all of them 2-free. Recall from Corollary \ref{soloZ} that in this situation $M_D \simeq \bigvee_{[b]\in B} Z^b_D$, thus we can restrict to the independent study of each connected pair of circles when computing $Z_D^b$. 

At this point we can make some simplifications in $D(\uno)$: Lemmas \ref{lemma:noregions1} and \ref{lemma:equivalent} allow us to remove nested monochords and equivalent bichords when computing $M_D$ (in exchange of taking suspensions). Moreover, by Lemma \ref{lemma:2free}\eqref{2free:3} we can assume that $D(\uno)$ contains no $b$-free monochords for any bichord $b$. These simplifications motivate the following definition:

\begin{df}\label{simple} A diagram $D$ is \emph{simple} if the associated chord diagram $D(\uno)$ contains at least two monochords, all of them 2-free, and satisfies the following conditions:
\begin{enumerate}
\item It contains exactly two circles;
\item It contains no nested monochords; 
\item There are no $b$-free monochords for any bichord $b$;
\item There are no equivalent bichords. 
\end{enumerate}
\end{df}

By definition, if $D$ is simple then $D(\uno)$ can be isotoped in $S^2$ so that monochords and bichords lie in different regions (see Figure \ref{Exsimple}). Moreover, since $D(\uno)$ contains no nested monochords, a circle with $n$ monochords is divided into $n$ half-disks (each of them bounded by a monochord and an arc of the circle) and an additional region (bounded by the $n$ monochords and $n$ arcs of the circle) that we call \emph{polygon}. 

Given a monochord $e$ in $D(\uno)$, we write $d_e$ for the half-disk bounded by $e$\footnote{If $e$ is the only monochord attached to the circle, then there are two possible options for $d_e$; in this case choose an option minimizing $|d_e|$, without loss of generality.}, and $|d_e|$ for the number of bichords having one of their endpoints in $d_e$. In simple diagrams $|d_e|\geq 1$ for any monochord $e$. 

In the proofs of Lemmas \ref{lemma:reduccion3}, \ref{lemma:disk-polygon} and \ref{lemma:disk-disk} we use several results from \cite[Section 3]{JACO}.

\begin{figure}[t]
\centering
\includegraphics[width = 9.5cm]{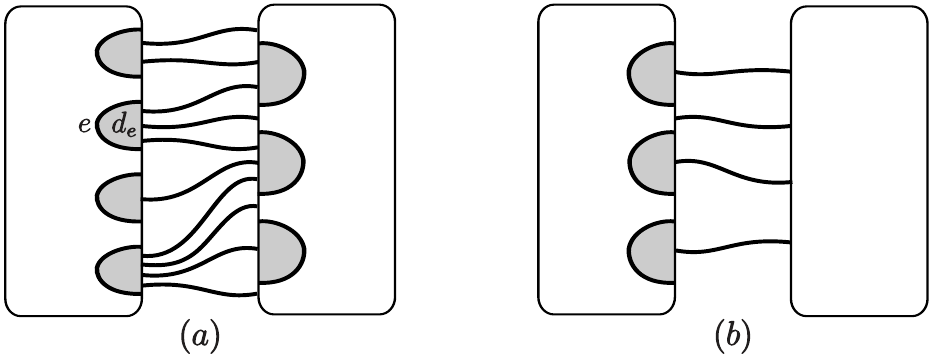}
\caption{\small{Two chord diagrams corresponding to simple diagrams. Grey regions correspond to half-disks bounded by monochords. In $(a)$ $|d_e|=3$, while $|d_{e'}|=1$ for every monochord $e'$ in $(b)$. }}
\label{Exsimple}
\end{figure}

\begin{lemma}\label{lemma:reduccion3} Let $D$ be a simple diagram whose associated chord diagram $D(\uno)$ contains a monochord $e$ so that $|d_e|=1$. Then $|M_D|$ is either contractible or homotopy equivalent to a sphere.
\end{lemma}

\begin{proof}
Let $b$ be the only bichord with and endpoint in $d_e$ and consider the skein sequence \eqref{skeinbichord} along $b$:
\[M_{D[b=1]}\lra X_{D[b=0]}\lra M_D.\]

\begin{figure}[h]
\centering
\includegraphics[width = 11cm]{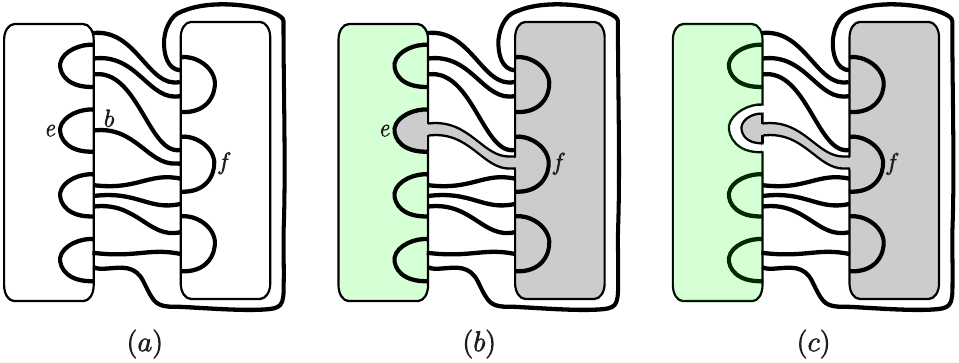}
\caption{\small{The $\uno$-resolutions corresponding to $D$, $D[b=0]$ and $D[b=0, e=0]$ illustrating the proof of Lemma \ref{lemma:reduccion3} are shown in $(a)$, $(b)$ and $(c)$, respectively. Regions $R$ and $R'$ are shaded green and grey, respectively.}}
\label{Discosprueba}
\end{figure}

The monochord $e$ is free in $D[b=1](\uno)$, thus by Lemma \ref{lemma:2free} $Z_{D[b=1]}$ is contractible, thus $M_{D[b=1]}$ is contractible (by Corollary \ref{soloZ}) and $M_D\simeq X_{D[b=0]}$.

Now, $D[b=0](\uno)$ has a single circle and its inner part is separated into two regions $R$ and $R'$ by the monochord $e$. See Figure \ref{Discosprueba}(b). Moreover:
\begin{enumerate}
\item monochords in $D(\uno)$ become internal monochords in $D[b=0](\uno)$;
\item\label{i:2} monochords attached to each of the circles in $D(\uno)$ become monochords contained in each of the regions $R, R'$ in $D[b=0](\uno)$;
\item\label{i:3} bichords in $D(\uno)$ become external monochords in $D[b=0](\uno)$ having one endpoint in $R$ and the other one in $R'$.
\end{enumerate}

Consider now the skein sequence \eqref{skeinX} along the monochord $e$
\[X_{D[b=0,e=1]}\lra X_{D[b=0,e=0]} \lra X_{D[b=0]}.\]
Condition \eqref{i:3} above implies that there are no alternating pairs in $D[b=0,e=0](\uno)$ (see Figure \ref{Discosprueba}(c)), so the middle term in the above sequence is contractible and we have $X_{D[b=0]}\simeq \Sigma X_{D[b=0,e=1]}$.

Assume first that $b$ connects $d_e$ with the region that we call polygon, then we claim that the Lando graph $G$ of $D[b=0,e=1]$, contains no cycles (i.e., it is a forest). In order to prove that, we proceed by contradiction. Suppose that $C_n$ is a cycle of length $n$ in $G$, with vertices $e_1, e_2, \ldots, e_n$ numbered so that $e_1$ is an internal chord in the region $R$. Conditions \eqref{i:2} and \eqref{i:3} imply that monochords $e_{4k+1}$ and $e_{4k+3}$ lie in $R$ and $R'$, respectively. Moreover, the absence of equivalent bichords in $D(\uno)$ implies the absence of cycles of length four in $G$, i.e., $n \geq 8$. The monochord $e_4$ separates the external region into two subregions $S,S'$. Since there are no cycles of length $4$, $e_2$ and $e_6$ lie one in each of these two subregions, say $e_2\in S$ and $e_6\in S'$. Moreover, as the internal monochords are non-nested, all internal monochords but $e_3$ and $e_5$ have both endpoints in either $S$ or $S'$. Therefore
\begin{enumerate}
\item the endpoints of $e_{1}$ lie in $S$ (because $e_2$ lies in $S$) and
\item the endpoints of $e_{2k+1}$ lie in $S'$ if $k\geq 3$ (because $e_{2k}$ lies in $S'$ if $k\geq 3$).
\end{enumerate}

In particular, both endpoints of $e_{n-1}$ lie in $S'$. This is a contradiction, since $e_n$ is a external monochord connecting $e_1$ and $e_{n-1}$. Therefore $G$ contains no cycles, and by \cite[Corollary~3.7]{JACO} its associated independence complex is either contractible or a sphere. Corollary \ref{corID=XD} completes the proof. 

Assume now that $b$ has its endpoints in two half-discs $d_e$ and $d_f$, for some monochord $f$, and consider the skein sequence \eqref{skeinX} along $f$. We get that $X_{D[b=0,e=1,f=0]}$ is contractible, since all monochords in $R'$ are 2-free. Therefore, $X_{D[b=0,e=1]}\simeq \Sigma X_{D[b=0,e=1,f=1]}$. The latter is the same as $X_{D'[b=0,e=1]}$, where $D' = D[f=1]$ is a diagram where $e$ connects $d_e$ with the polygon, which was studied in the previous case.
\end{proof}

\begin{remark}
Since simple diagrams do not contain equivalent bichords, if $|d_e| \geq 4$ for some monochord $e$, then there exists a monochord $f$ with $|d_f|=1$ ($d_f$ is connected to $d_e$ by a bichord). In other words, if $|d_e| > 1$ for all monochords $e$, then $|d_e |\leq 3$ for all monochords $e$.    
\end{remark}

\begin{df}
A simple diagram $D$ so that $2 \leq |d_e| \leq 3$ for every monochord $e$ of $D(\uno)$ is called \textit{super-simple}. 
\end{df}

\begin{figure}[t]
\centering
\includegraphics[width = 10cm]{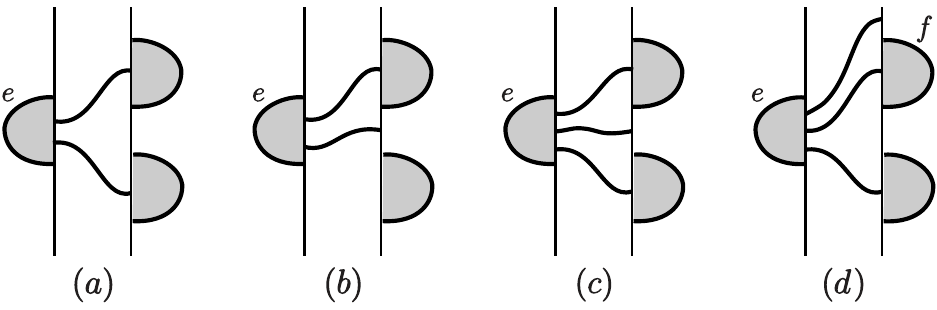}
\caption{\small{The three possible situations for a half-disk $d_e$ in the 1-resolution of a super-simple diagram are illustrated in Figures (a), (b) and (c). Situation (d) is not allowed in semi-simple diagrams.}}
\label{Casosdiscos}
\end{figure}

If $D$ is super-simple, every half-disk $d_e$ in $D(\uno)$ satisfies one of the three following conditions:
\begin{enumerate}
\item $|d_e|=2$, and both bichords connect $d_e$ with two contiguous half-disks in the other circle (see Figure \ref{Casosdiscos}(a));
\item $|d_e|=2$, and one bichord connects $d_e$ with a half-disk in the other circle and the other one connects $d_e$ with the polygon in the other circle (see Figure \ref{Casosdiscos}(b));
\item $|d_e|=3$, and two bichords connect $d_e$ with two contiguous half-disks in the other circle, and the third one is placed between them and connects $d_e$ with the polygon of the other circle (see Figure \ref{Casosdiscos}(c)).
\end{enumerate} 

The next statement follows from Lemmas \ref{lemma:disk-polygon} and \ref{lemma:disk-disk}.

\begin{proposition}\label{prop_supersimple}
Let $D$ be a super-simple diagram. Then $|M_D|$ is homotopy equivalent to a wedge of spheres.
\end{proposition}

\begin{lemma}\label{lemma:disk-polygon}
Let $D$ be a super-simple diagram so that $D(\uno)$ contains at least one bichord connecting a half-disk with the region called polygon. Then $|M_D|$ is either contractible or homotopy equivalent to a sphere.
\end{lemma}

\begin{proof}
Let $b$ be a bichord connecting a half-disk $d_e$ in a circle $z_1$ with the polygon of the circle $z_2$, for some monochord $e$.  See Figure \ref{Discopoligono}$(a)$. Since $D$ is super-simple, there exists at least a bichord $a$ connecting $d_e$ with $d_f$ for a monochord $f$ attached to $z_2$. We distinguish two cases, depending on whether $z_2$ contains at least two monochords or it contains just a single monochord. 

\textit{Case 1:} Assume that the chord diagram $D(\uno)$ contains more than one monochord attached to $z_2$, and consider the skein exact sequence \eqref{skeinbichord} along $a$:
$$
M_{D[a=1]}\lra X_{D[a=0]}\lra M_D. 
$$

\begin{figure}[t]
\centering
\includegraphics[width = 11.5cm]{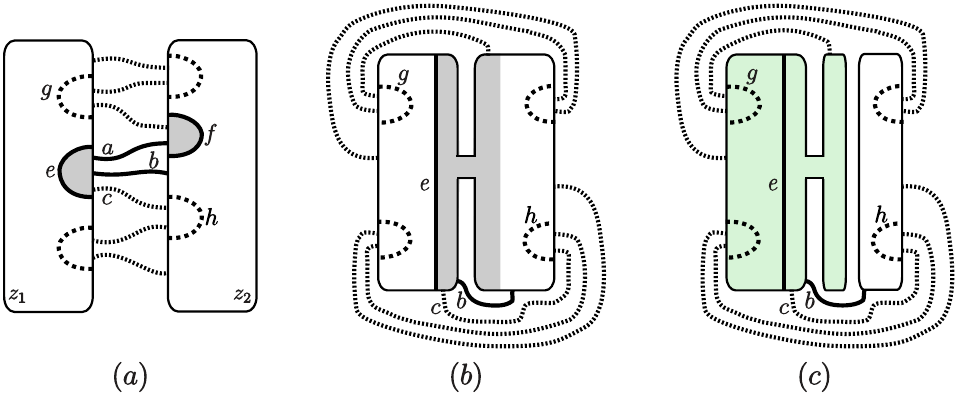}
\caption{\small{The $\uno$-resolutions corresponding to $D$, $D[a=0, f=1]$ and $D[a=0, f=0]$ illustrating the proof of Lemma \ref{lemma:disk-polygon} are shown in $(a), (b)$, and $(c)$, respectively. Dotted lines represent those chords that may or may not be in the diagrams.}}
\label{Discopoligono}
\end{figure}

We will show that $X_{D[a=0]}$ is contractible. To do so, we apply the skein sequence \eqref{skeinX} along the monochord $f$: 
$$
X_{D[a=0, f=1]}\lra X_{D[a=0, f=0]} \lra X_{D[a=0]}. 
$$

Figure \ref{Discopoligono}$(b)$ represents $D[a=0, f=1](\uno)$, where $b$ is a free monochord. Moreover, any monochord different from $f$ attached to $z_2$ becomes 2-free in $D[a=0, f=0](\uno)$, as illustrated in Figure \ref{Discopoligono}$(c)$. Therefore by Lemma \ref{lemma:2free}(1) $X_{D[a=0, f=1]}$ and $X_{D[a=0, f=0]}$ are contractible, and so is $X_{D[a=0]}$. 

As a consequence, $M_D \simeq \Sigma M_{D[a=1]}$. If $|d_e| = 2$ in $D(\uno)$, then $|d_e|=1$ in $D[a=1](\uno)$, and the statment holds by Lemma \ref{lemma:reduccion3}. If $|d_e| = 3$ in $D(\uno)$, then $|d_e|=2$ in $D[a=1](\uno)$ and it contains a bichord $c$ connecting $d_e$ with $d_h$ for a monochord $h$ (contiguous to $f$) attached to the circle $z_2$. Applying the skein exact sequence \eqref{skeinbichord} along $c$ and repeating the same reasoning as before leads to $M_D \simeq \Sigma^2 M_{D[a=1,c=1]}$. Since $|d_e|=1$ in $D[a=1,c=1](\uno)$,  Lemma \ref{lemma:reduccion3} completes the proof for this case. 

\textit{Case 2:} Assume that there is just one monochord $f$ attached to the circle $z_2$ in $D(\uno)$. Then, since $D$ is super-simple, $D(\uno)$ is as one of the six chord diagrams depicted in Figure \ref{Solo1monocuerda} (each picture leads to two possible chord diagrams, depending on whether it contains the bichord $p$ or not). In order to compute $M_D$ for each of these situations, we use skein sequence \eqref{skeinbichord} together with Corollary \ref{corID=XD}. We also use \cite[Lemma 3.2, Corollary 3.3]{JACO} to compute $I_D$: 

\begin{itemize}
\item [-] If $D(\uno)$ is as in Figure \ref{Solo1monocuerda}$(a)$ with $p$: we consider the skein sequence \eqref{skeinbichord} along the bichord $c$, and get that $|X_{D[c=0]}| \simeq S^2$. Next, we show that $M_{D[c=1]}$ is contractible by applying the skein sequence \eqref{skeinbichord} to $D[c=1]$ along the bichord $a$ (we use the fact that chord $f$ is free in $D_{[c=1, a=1]}(\uno)$). Therefore, $|M_D| \simeq |X_{D[c=0]}| \simeq S^2$.

\item [-] If $D(\uno)$ is as in Figure \ref{Solo1monocuerda}$(a)$ without $p$: we proceed as in the previous case, and get that both $M_{D[c=1]}$ and $X_{D[c=0]}$ are contractible, and so is $M_D$.

 \item [-] If $D(\uno)$ is as in Figure \ref{Solo1monocuerda}$(b)$ with $p$: we consider the skein sequence \eqref{skeinbichord} along the bichord $a$, and get that $|X_{D[a=0]}| \simeq S^4$. Next, we show that $M_{D[a=1]}$ is contractible by applying the skein sequence \eqref{skeinbichord} to $D[a=1]$ along the bichord $b$ (we use the fact that chord $e$ is free in $D_{[a=1, b=1]}(\uno)$). Therefore, $|M_D| \simeq |X_{D[a=0]}| \simeq S^4$.
 
\item [-] If $D(\uno)$ is as in Figure \ref{Solo1monocuerda}$(b)$ without $p$: the procedure is analogous to the previous one, the only difference is that $|X_{D[a=0]}| \simeq S^3$ and therefore $|M_D| \simeq S^3$.

\item [-] If $D(\uno)$ is as in Figure \ref{Solo1monocuerda}$(c)$ with $p$: we consider the skein sequence \eqref{skeinbichord} along the bichord $a$, and get that $X_{D[a=0]}$ is contractible. Next, we show that $M_{D[a=1]}$ is also contractible by applying the skein sequence \eqref{skeinbichord} to $D[a=1]$ along the bichord $b$ (we use the fact that chord $e$ is free in $D_{[a=1, b=1]}(\uno)$). Therefore, $M_D$ is contractible. 

 \item [-] If $D(\uno)$ is as in Figure \ref{Solo1monocuerda}$(c)$ without $p$: we proceed as in the previous case, and get that $X_{D[a=0]}$ is contractible and $|M_{D[a=1]}| \simeq S^3$. Therefore, $|M_D| \simeq \Sigma |M_{D[a=1]}| \simeq S^4$. \qedhere
\end{itemize}
 
\begin{figure}[t]
\centering
\includegraphics[width = 11.5cm]{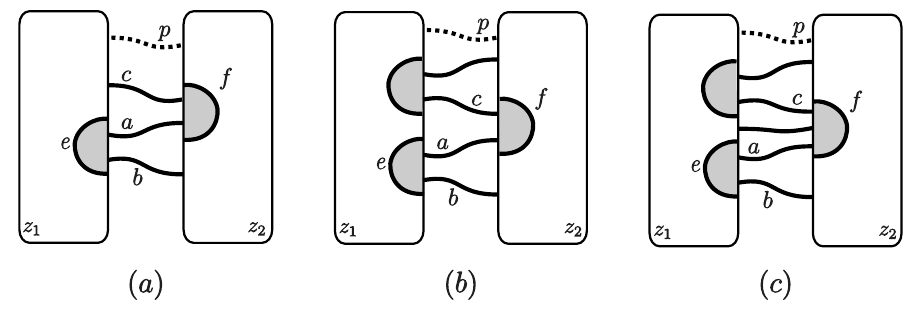}
\caption{\small{The six possible chord diagrams illustrating \textit{Case 2} of the proof of Lemma \ref{lemma:disk-polygon}.}}
\label{Solo1monocuerda}
\end{figure}

\end{proof}

\begin{lemma}\label{lemma:disk-disk}
Let $D$ be a super-simple diagram so that all bichords in $D(\uno)$ connect two half-disks (i.e., there are no bichords with an endpoint in the region called polygon). Then, $|M_D|$ is homotopy equivalent to a wedge of spheres. More precisely, if we write $2n$ for the number of monochords in $D$, then 
$$|M_D| \simeq \begin{cases}
S^{\frac{8n}{3}-1} \vee S^{\frac{8n}{3}-1} & \text{if  } n \equiv 0 \text{ (mod 3)}, \\
S^{\frac{8n+1}{3}-1}  & \text{if  } n \equiv 1 \text{ (mod 3)}, \\
S^{\frac{8n+2}{3}-2}  & \text{if  } n \equiv 2 \text{ (mod 3)}.
\end{cases}$$
\end{lemma}

\begin{proof}
First, notice that since all bichords in $D(\uno)$ connect two half-disks, the number of monochords equals the number of bichords. We label the chords as follows (see Figure \ref{Supersimpl}$(a)$): we write $e_1, \ldots, e_n$ (resp. $f_1, \ldots, f_n)$ for the monochords attached to the circle $z_1$ (resp. $z_2$), and $a_i$ (resp. $b_i$) for the bichord having its endpoints in the half-disks $d_{e_i}$ and $d_{f_i}$ (resp. $d_{e_{i+1}}$ and $d_{f_i}$), for $1\leq i \leq n$.

Consider the skein sequence \eqref{skeinbichord} along the bichord $a_1$:
\begin{equation}\label{proofnotpolygon}
M_{D[a_1=1]} \lra X_{D[a_1=0]} \lra M_{D}. 
\end{equation}

We study now the homotopy type of $X_{D[a_1=0]}$. Figure \ref{Supersimpl}$(b)$ shows the chord diagram $D[a_1=0](\uno)$, whose associated Lando graph $G= G_L(D[a_1=0])$ is as depicted in Figure \ref{Supersimpl}(c). Now, since vertex $e_1$ dominates\footnote{Following \cite{JACO}, given two vertices $v$ and $w$ in a graph, we say that $v$ dominates $w$ if the adjacent vertices to $w$ are also adjacent to $v$.} $e_2$ and vertex $f_1$ dominates $f_2$, it follows from \cite[Lemma 3.2]{JACO} that the independence complex associated to $G_L$ is homotopy equivalent to the independence complex associated to the graph $G-e_1-f_1$, which is a path of length $4n-4$, $L_{4n-4}$, whose independence complex follows from \cite[Corollary 3.4]{JACO}:
$$|I_G| \simeq |I_{G-e_1-f_1}| = |I_{L_{4n-4}}| \simeq \begin{cases}
* & \text{if  } n \equiv 1 \mod 3, \\
S^{\lfloor \frac{4n-4}{3}\rfloor} & \text{otherwise},
\end{cases}$$

and by Corollary \ref{corID=XD} we get:
$$|X_{D[a_1=0]}| \simeq \begin{cases}
* & \text{if  } n \equiv 1 \mod 3, \\
S^{4n-3 - \lfloor \frac{4n-4}{3}\rfloor} & \text{otherwise}.
\end{cases}$$

\begin{figure}[t]
\centering
\includegraphics[width = 11.5cm]{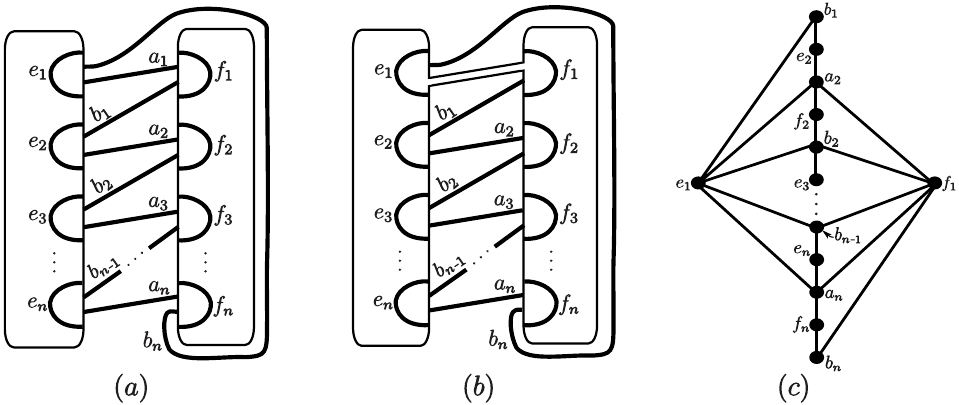}
\caption{\small{The chord diagram $D(\uno)$ satisfying conditions in Lemma \ref{lemma:disk-disk} is shown in $(a)$. The chord diagram and the Lando graph associated to $D[a_1=0]$ are shown in $(b)$ and $(c)$, respectively.}}
\label{Supersimpl}
\end{figure}

Next, we compute the homotopy type of $M_{D[a_1=1]}$. Consider the skein sequence \eqref{skeinbichord} along the bichord $b_1$:
$$
M_{D[a_1=1, b_1=1]} \lra X_{D[a_1=1, b_1=0]} \lra M_{D[a_1=1]}. 
$$

\begin{figure}[t]
\centering
\includegraphics[width = 11.5cm]{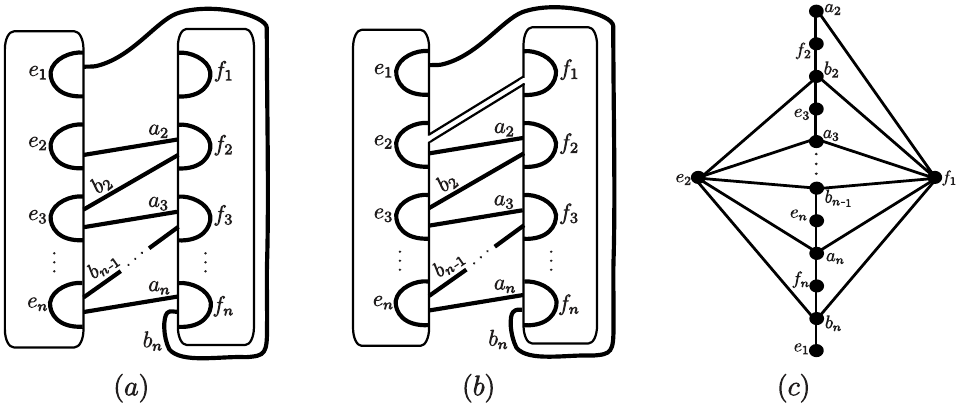}
\caption{\small{The chord diagram associated to $D[a_1=1,b_1=1]$ is shown in $(a)$. The chord diagram and the Lando graph associated to $D[a_1=1,b_1=0]$ are shown in $(b)$ and $(c)$, respectively.}}
\label{Supersimpl2}
\end{figure}

Figures \ref{Supersimpl2}$(a)$ and $(b)$ represent the chord diagrams associated to $D[a_1=1, b_1=1]$ and $D[a_1=1, b_1=0]$, respectively. Notice that $f_1$ is free and $e_1$ is $2$-free in $D[a_1=1, b_1=1](\uno)$, hence $M_{D[a_1=1, b_1=1]}$ is contractible by Lemma \ref{lemma:2free} and $M_{D[a_1=1]} \simeq X_{D[a_1=1, b_1=0]}$. The Lando graph $H$ associated to $D[a_1=1, b_1=0]$ is depicted in Figure \ref{Supersimpl2}$(c)$, and since vertex $f_1$ dominates $f_2$ and vertex $e_2$ dominates $e_3$, we have the following equivalence relations:
$$|I_H| \simeq |I_{H-e_2-f_1}| = |I_{L_{4n-5}}| \simeq \begin{cases}
* & \text{if  } n \equiv 2 \mod 3, \\
S^{\lfloor \frac{4n-5}{3}\rfloor} & \text{otherwise},
\end{cases}$$

and by Corollary \ref{corID=XD} we get:
$$|M_{D[a_1=1]}| \simeq |X_{D[a_1=1, b_1=0]}| \simeq \begin{cases}
* & \text{if  } n \equiv 2 \mod 3, \\
S^{4n-4 - \lfloor \frac{4n-5}{3}\rfloor} & \text{otherwise}.
\end{cases}$$

Finally, we compute $|M_D|$: 
\begin{itemize}
\item[-] If $n \equiv 0$ (mod $3$), then the skein sequence \eqref{proofnotpolygon} becomes: 
$$S^{\frac{8n}{3}-2} \lra S^{\frac{8n}{3}-1} \lra |M_{D}|,$$
so $|M_D| \simeq  S^{\frac{8n}{3}-1} \vee S^{\frac{8n}{3}-1}.$

\item[-] If $n \equiv 1$ (mod $3$), then the skein sequence \eqref{proofnotpolygon} becomes: 
$$S^{\frac{8n+1}{3}-2} \lra \ast \lra |M_{D}|,$$
so $|M_D| \simeq S^{\frac{8n+1}{3}-1}$.

\item[-] If $n \equiv 2$ (mod $3$), then the skein sequence \eqref{proofnotpolygon} becomes: 
$$\ast \lra S^{\frac{8n+2}{3}-2} \lra |M_{D}|,$$
so $|M_D| \simeq S^{\frac{8n+2}{3}-2}$.
\end{itemize}
\end{proof}

\begin{theorem}
Let $D$ be a diagram so that the associated chord diagram $D(\uno)$ contains more than one monochord, all of them 2-free . Then $|M_D|$ is homotopy equivalent to a wedge of spheres.
\end{theorem}

\begin{proof}
Since all monochords are $2$-free, then $M_D \simeq \bigvee_{[b]\in B} Z^b_D$, by Corollary \ref{soloZ}, so we consider each pair of discs in $D(\uno)$ independently. Moreover, we can remove nested monochords and equivalent bichords when computing $M_D$ at the expense of taking suspensions (Lemmas \ref{lemma:noregions1} and \ref{lemma:equivalent}). In addition, we can assume that $D(\uno)$ contains no $b$-free monochords for any bichord $b$ (otherwise, $Z_D^b$ is contractible by Lemma \ref{lemma:2free}\eqref{2free:3}). Hence, we just need to prove the statement for simple diagrams. Lemma \ref{lemma:reduccion3} and Proposition \ref{prop_supersimple} complete the proof. 
\end{proof}


\bibliographystyle{amsalpha}
\bibliography{biblio-article}

\end{document}